\documentclass[11pt,a4paper,english]{article}

\usepackage[latin1]{inputenc}
\usepackage[sort,numbers]{natbib}
\usepackage{babel}
\usepackage[centertags]{amsmath}
\usepackage{enumerate}
\usepackage{amsfonts}
\usepackage{amssymb}
\usepackage{mathtools}
\usepackage{amsthm}
\usepackage{newlfont}
\usepackage{dsfont}
\usepackage{geometry}
\usepackage{mathrsfs}
\usepackage{authblk}
\usepackage{verbatim}
\usepackage[hidelinks]{hyperref}

\DeclareMathOperator*{\esssup}{ess\,sup}

\newtheorem{theorem}{Theorem}[section]
\newtheorem{lemma}[theorem]{Lemma}

\newtheorem{proposition}[theorem]{Proposition}
\newtheorem{definition}[theorem]{Definition}
\newtheorem{remark}[theorem]{Remark}
\newtheorem{assumption}[theorem]{Assumption}

\newcommand{\msc}[1]{\textbf{MSC2010 Classification:} #1.}

\newcommand{\keywords}[1]{\textbf{Key words:} #1.}

\makeatletter  
\@addtoreset{equation}{section}

\makeatother  
\begin{document}
\title{\textbf{
Finite-horizon optimal multiple switching with signed switching costs}\footnote{This research was partially supported by EPSRC grant EP/K00557X/1.}}

\author{Randall Martyr\footnote{School of Mathematics, The University of Manchester, Oxford Road, Manchester M13 9PL, United Kingdom. email: \texttt{randall.martyr@postgrad.manchester.ac.uk}}}
\maketitle
\begin{abstract}
This paper is concerned with optimal switching over multiple modes in continuous time and on a finite horizon. The performance index includes a running reward, terminal reward and switching costs that can belong to a large class of stochastic processes. Particularly, the switching costs are modelled by right-continuous with left-limits processes that are quasi-left-continuous and can take both positive and negative values. We provide sufficient conditions leading to a well known probabilistic representation of the value function for the switching problem in terms of interconnected Snell envelopes. We also prove the existence of an optimal strategy within a suitable class of admissible controls, defined iteratively in terms of the Snell envelope processes.
\end{abstract}
\msc{93E20, 60G40, 91B99, 62P20}
\vspace{+4pt}

\noindent\keywords{optimal switching, real options, stopping times, optimal stopping problems, Snell envelope} 
\section{Introduction.}
The recent paper by Guo and Tomecek \cite{Guo2008b} showed a connection between Dynkin games and optimal switching problems with signed (positive and negative) switching costs. The results were obtained for a model in which the cost/reward processes were merely required to be adapted and satisfy mild integrability conditions. However, there are few theoretical results on the existence of optimal switching control policies under such general conditions.

Optimal switching for models driven by discontinuous stochastic processes has been studied previously in papers such as \cite{Hamadene2007a,Morimoto1987}. The paper \cite{Morimoto1987} used optimal stopping theory to study the optimal switching problem on an infinite time horizon with multiple modes. The model described in \cite{Morimoto1987} has bounded and non-negative running rewards which are driven by right-continuous processes, and switching costs that are strictly positive and constant. The paper \cite{Hamadene2007a} studied the finite time horizon optimal switching problem with two modes. The model has running rewards adapted to a filtration generated by a Brownian motion and an independent Poisson random measure, but excludes terminal data and assumes switching costs that are strictly positive and constant. The more recent paper \cite{Djehiche2009} has a model similar to \cite{Hamadene2007a} with switching costs assumed to be continuous stochastic processes and filtration generated by a Brownian motion. Nevertheless, the authors stated (\cite[p.~2753]{Djehiche2009}) that their results can be adapted to a more general setup, possibly by using the same approach as in \cite{Hamadene2007a}.

Most of the literature on optimal switching assumes non-negative switching costs. However, signed switching costs are important in models where the controller can (partially) recover its investment, or receive a subsidy/grant for investing in a new technology such as renewable (green) energy production \cite{Guo2008b,Lumley2001,LyVath2008}. The preprint \cite{ElAsri2012} sought to generalise the results of \cite{Djehiche2009} by permitting signed switching costs, but at the expense of limiting the total number of switches incurring negative costs. This limitation is absent in papers such as \cite{Bouchard2009,Lundstrom2013a} where the optimal switching problem was studied within a Markovian setting. There are, however, other structural conditions and hypotheses made in \cite{Bouchard2009,Lundstrom2013a}. For example, there is an assumption in \cite[p.~1221]{Lundstrom2013a} that the terminal reward is the same for all modes (which also implies the terminal values of the switching costs are non-negative).

The Snell envelope approach, also known as the \emph{method of essential supremum} \cite{Peskir2006}, is a general approach to optimal stopping problems which does not require Markovian assumptions on the data. It was used in the aforementioned papers \cite{Hamadene2007a,Djehiche2009} and the paper \cite{Bayraktar2010} on optimal switching problems for one-dimensional diffusions (albeit in a slightly different manner). In this paper we use the theory of Snell envelopes to extend Theorems 1 (verification) and 2 (existence) of \cite{Djehiche2009}. Our model allows for non-zero terminal data and switching costs which are real-valued stochastic processes with paths that are right-continuous with left limits and \emph{quasi-left-continuous}.

In contrast to \cite{ElAsri2012}, our results do not presuppose a limit on the total number of switches incurring negative costs. We do, however, require that a certain ``martingale hypothesis'' $\textbf{M}$ on the switching costs be satisfied (see Section~\ref{Section:ArbitraryNumberOfSwitches} below). This hypothesis can be verified in many cases of interest, including the case of two modes, and does not require the terminal reward to be the same for all modes. We also assume that the filtration, in addition to satisfying the usual conditions of right-continuity and completeness, is \emph{quasi-left-continuous}. This property, which generalises the assumption made in \cite{Hamadene2007a}, is satisfied in many applications. For example, it holds when the filtration is the natural (completed) filtration of a L\'{e}vy process. Such models have wide ranging applications in finance, insurance and control theory \cite{Oksendal2007}.

The layout of paper is as follows. Section~\ref{Section:Optimal-Switching-Definitions} introduces the probabilistic model and optimal switching problem. Preliminary concepts from the general theory of stochastic processes and optimal stopping are recalled in Section~\ref{Section:Optimal-Switching-Preliminaries}. The modelling assumptions for the optimal switching problem are given in Section~\ref{Section:Optimal-Switching-Assumptions}. A verification theorem establishing the relationship between the optimal switching problem and iterative optimal stopping is given in Section~\ref{Section:Optimal-Switching-VerificationTheorem}. Sufficient conditions for validating the verification theorem's hypotheses are discussed in Section~\ref{Section:Optimal-Switching-Existence-Proof}. The conclusion, appendix, acknowledgements and references then follow.

\section{Definitions.}\label{Section:Optimal-Switching-Definitions}
\subsection{Probabilistic setup.}
We work on a time horizon $[0,T]$, where $0 < T < \infty$. It is assumed that a complete filtered probability space, $\left(\Omega,\mathcal{F},\mathbb{F},\mathsf{P}\right)$, has been given and the filtration $\mathbb{F} = \left(\mathcal{F}_{t}\right)_{0 \le t \le T}$ satisfies the usual conditions of right-continuity and augmentation by the $\mathsf{P}$-null sets. Let $\mathsf{E}$ denote the corresponding expectation operator. We use $\mathbf{1}_{A}$ to represent the indicator function of a set (event) $A$. The shorthand notation a.s. means ``almost surely''. Let $\mathcal{T}$ denote the set of $\mathbb{F}$-stopping times $\nu$ which satisfy $0 \le \nu \le T$ \hspace{1bp} $\mathsf{P}$-a.s. For a given $S \in \mathcal{T}$, write $\mathcal{T}_{S} = \{\nu \in \mathcal{T} \colon  \nu \ge S \enskip \mathsf{P}-a.s.\}$. Unless otherwise stated, a stopping time is assumed to be defined with respect to $\mathbb{F}$. For notational convenience the dependence on $\omega \in \Omega$ is often suppressed.
\subsection{Problem definition.}
The controller in an optimal switching problem influences a dynamical system over the horizon $[0,T]$ by choosing operating modes from a finite set $\mathbb{I} = \lbrace 1,\ldots,m \rbrace$ with $m \ge 2$. The instantaneous profit in mode $i \in \mathbb{I}$ is a mapping $\psi_{i} \colon \Omega \times [0,T] \to \mathbb{R}$. There is a cost for switching from mode $i$ to $j$ which is given by $\gamma_{i,j} \colon \Omega \times [0,T] \to \mathbb{R}$. There is also a reward for being in mode $i \in \mathbb{I}$ at time $T$, denoted by $\Gamma_{i}$, which is a real-valued random variable. The assumptions on these costs / rewards are discussed below in Section~\ref{Section:Optimal-Switching-Assumptions}.

\begin{definition}[Admissible Switching Control Strategies]\label{Definition:Optimal-Switching-AdmissibleStrategies}
	Let $t \in [0,T]$ and $i \in \mathbb{I}$ be given. An admissible switching control strategy starting from $(t,i)$ is a double sequence $\alpha = \left(\tau_{n},\iota_{n}\right)_{n \ge 0}$ of stopping times $\tau_{n} \in \mathcal{T}_{t}$ and mode indicators $\iota_{n}$ such that:
	\begin{enumerate}
		\item $\tau_{0} = t$ and the sequence $\lbrace \tau_{n} \rbrace_{n \ge 0}$ is non-decreasing;
		\item Each $\iota_{n} \colon \Omega \to \mathbb{I}$ is $\mathcal{F}_{\tau_{n}}$-measurable; $\iota_{0} = i$ and $\iota_{n} \neq \iota_{n+1}$ for $n \ge 0$;
		\item Only a finite number of switching decisions can be made before the terminal time $T$:
		\begin{equation}\label{eq:Optimal-Switching-FiniteStrategy}
		\mathsf{P}\left(\{\tau_{n} < T,\hspace{1bp} \forall n \ge 0\}\right) = 0.
		\end{equation}
		\item The family of random variables $\{ C^{\alpha}_{n}\}_{n \ge 1}$, where $C^{\alpha}_{n}$ is the total cost of the first $n \ge 1$ switches
		\begin{equation*}\label{eq:Optimal-Switching-CumulativeSwitchingCost}
		C^{\alpha}_{n} \coloneqq \sum\limits_{k = 1}^{n}\gamma_{\iota_{k-1},\iota_{k}}(\tau_{k})\mathbf{1}_{\{\tau_{k} < T\}}
		\end{equation*}
		satisfies
		\begin{equation}\label{eq:Optimal-Switching-UISupremumClosureSwitchingCosts}
		\mathsf{E}\big[\sup_{n}\big|C^{\alpha}_{n}\big|\big] < \infty.
		\end{equation}
	\end{enumerate}
	Let $\mathcal{A}_{t,i}$ denote the set of admissible switching control strategies (henceforth, just strategies). We write $\mathcal{A}_{i}$ when $t = 0$ and drop the subscript $i$ if it is not important for the discussion.
\end{definition}

\begin{remark}\label{Remark:Optimal-Switching-SubscriptModeIndicators}
	Processes or functions with super(sub)-scripts in terms of the random mode indicators $\iota_{n}$ are interpreted in the following way:
	\begin{align*}
	Y^{\iota_{n}} & = \sum\limits_{j \in \mathbb{I}}\mathbf{1}_{\lbrace \iota_{n} = j \rbrace} Y^{j}, \quad n\ge 0 \\
	\gamma_{\iota_{n-1},\iota_{n}}\left(\cdot\right) & = \sum\limits_{j \in \mathbb{I}}\sum\limits_{k \in \mathbb{I}}\mathbf{1}_{\lbrace \iota_{n-1} = j \rbrace}\mathbf{1}_{\lbrace \iota_{n} = k \rbrace} \gamma_{j,k}\left(\cdot\right), \quad n \ge 1.
	\end{align*}
	Note that the summations are finite.
\end{remark}

We shall frequently use the notation $N(\alpha)$ to denote the (random) number of switches before $T$ under strategy $\alpha$:
\begin{equation}\label{eq:Optimal-Switching-RandomNumberOfSwitches}
N(\alpha) = \sum_{n \ge 1}\mathbf{1}_{\{\tau_{n} < T\}}, \quad \alpha \in \mathcal{A}.
\end{equation}
Associated with each strategy $\alpha \in \mathcal{A}$ is a mode indicator function $\mathbf{u} \colon \Omega \times [0,T] \to \mathbb{I}$ that gives the active mode at each time \cite{Djehiche2009,Guo2008b}:
\begin{equation}\label{eq:Optimal-Switching-ModeIndicatorContinuousTime}
\mathbf{u}_{t} \coloneqq \iota_{0}\mathbf{1}_{[\tau_{0},\tau_{1}]}(t) + \sum\limits_{n \ge 1}\iota_{n}\mathbf{1}_{(\tau_{n},\tau_{n+1}]}(t),\hspace{1em} t \in [0,T].
\end{equation}

For a fixed time $t \in [0,T]$ and given mode $i \in \mathbb{I}$, the performance index for the optimal switching problem starting at $t$ in mode $i$ is given by:
\begin{equation}\label{eq:Optimal-Switching-DynamicPerformanceIndex}
J(\alpha;t,i) = \mathsf{E}\left[\int_{t}^{T}\psi_{\mathbf{u}_{s}}(s){d}s + \Gamma_{\mathbf{u}_{T}} - \sum_{n \ge 1}\gamma_{\iota_{n-1},\iota_{n}}(\tau_{n})\mathbf{1}_{\{ \tau_{n} < T \}} \biggm \vert \mathcal{F}_{t}\right],\hspace{1em} \alpha \in \mathcal{A}_{t,i}.
\end{equation}
The goal is to find a strategy $\alpha^{*} \in \mathcal{A}_{t,i}$ that maximises the performance index:
\begin{equation}\label{eq:Optimal-Switching-ValueFunction}
J(\alpha^{*};t,i) = \esssup\limits_{\alpha \in \mathcal{A}_{t,i}}J\left(\alpha;t,i\right) \eqqcolon V(t,i).
\end{equation}
The random function $V(t,i)$ is called the \emph{value function} for the optimal switching problem.

\begin{remark}\label{Note:IntegrabilityConcerns}
	For $\alpha \in \mathcal{A}$, define $C^{\alpha}$ to be the total switching cost under $\alpha$:
	\[
	C^{\alpha} \coloneqq \sum\limits_{n \ge 1}\gamma_{\iota_{n-1},\iota_{n}}(\tau_{n})\mathbf{1}_{\{\tau_{n} < T\}}
	\]
	By the finiteness condition~\eqref{eq:Optimal-Switching-FiniteStrategy}, we have $\forall \alpha \in \mathcal{A}$:
	\[
	C^{\alpha} = \lim_{n \to \infty} C^{\alpha}_{n} \enskip \mathsf{P}-\text{a.s.}
	\]
	Furthermore, by using condition \eqref{eq:Optimal-Switching-UISupremumClosureSwitchingCosts}, we have the following dominated convergence property:
	\begin{equation}\label{eq:Optimal-Switching-ConvergenceOfSwitchingCost}
	\forall \alpha \in \mathcal{A}: \quad \lim_{n \to \infty}\mathsf{E}\left[C^{\alpha}_{n}\vert \mathcal{B}\right] = \mathsf{E}\left[C^{\alpha}\vert \mathcal{B}\right] \hspace{1em} \text{a.s. for every } \sigma\text{-algebra} \hspace{5bp} \mathcal{B} \subset \mathcal{F}.
	\end{equation}
\end{remark}

\section{Preliminaries.}\label{Section:Optimal-Switching-Preliminaries}
\subsection{Some results from the general theory of stochastic processes.}
We need to recall a few results from the general theory of stochastic processes that are essential to this paper. For more details the reader is kindly referred to the references \cite{Elliott1982,Jacod2003,Rogers2000b}.
\subsubsection{Right-continuous with left-limits processes.}
An adapted process $X = \left(X_{t}\right)_{0 \le t \le T}$ is said to be c\`{a}dl\`{a}g if it is right-continuous and admits left limits. The left-limits process associated with a c\`{a}dl\`{a}g process $X$ is denoted by $X_{-} = \left(X_{t^{-}}\right)_{0 < t \le T}$ (here we follow the convention of \cite{Rogers2000b}). We also define the process $\triangle X$ by $\triangle X \coloneqq X - X_{-}$ and let $\triangle_{t}X \coloneqq X_{t} - X_{t^{-}}$ denote the size of the jump in $X$ at $t \in (0,T]$.
\subsubsection{Predictable random times.}
A random time $S$ is an $\mathcal{F}$-measurable mapping $S \colon \Omega \to [0,T]$. For two random times $\rho$ and $\tau$, the stochastic interval $[\rho,\tau]$ is defined as:
\[
[\rho,\tau] = \left\lbrace (\omega,t) \in \Omega \times [0,T] \colon \rho(\omega) \le t \le \tau(\omega)  \right\rbrace.
\]
Stochastic intervals $(\rho,\tau]$, $[\rho,\tau)$, $(\rho,\tau)$ are defined analogously. A random time $S > 0$ is said to be predictable if the stochastic interval $[0,S)$ is measurable with respect to the predictable $\sigma$-algebra (the $\sigma$-algebra on $\Omega \times (0,T]$ generated by the adapted processes with paths that are left-continuous with right-limits on $(0,T]$). Note that every predictable time is a stopping time \cite[p.~17]{Jacod2003}. By Meyer's previsibility (predictability) theorem (\cite{Rogers2000b}, Theorem~\MakeUppercase{\romannumeral 6}.12.6), a stopping time $S > 0$ is predictable if and only if it is announceable in the following sense: there exists a sequence of stopping times $\lbrace S_{n} \rbrace_{n \ge 0}$ satisfying $S_{n}(\omega) \le S_{n+1}(\omega) < S(\omega)$ for all $n$ and $\lim_{n}S_{n}(\omega) = S(\omega)$.

\subsubsection{Quasi-left-continuous processes and filtrations.}
A c\`{a}dl\`{a}g process $X$ is called quasi-left-continuous if $\triangle_{S} X = 0$ a.s. for every predictable time $S$ (Definition \MakeUppercase{\romannumeral 1}.2.25 of \cite{Jacod2003}). The strict pre-$S$ $\sigma$-algebra associated with a random time $S > 0$,  $\mathcal{F}_{S^{-}}$, is defined as \cite[p.~345]{Rogers2000b}:
\[
\mathcal{F}_{S^{-}} = \sigma\left(\lbrace A \cap \lbrace S > u \rbrace \colon 0 \le u \le T, A \in \mathcal{F}_{u}  \rbrace\right).
\]
According to \cite[p.~346]{Rogers2000b}, a filtration $\mathbb{F} = (\mathcal{F}_{t})_{0 \le t \le T}$ which satisfies the usual conditions is said to be \emph{quasi-left-continuous} if $\mathcal{F}_{S} = \mathcal{F}_{S^{-}}$ for every predictable time $S$. We have the following equivalence result for quasi-left-continuous filtrations (see \cite{Rogers2000b}, Theorem \MakeUppercase{\romannumeral 6}.18.1 and \cite{Elliott1982}, Theorem 5.36).
\begin{proposition}[Characterisation of quasi-left-continuous filtrations]\label{Proposition:Optimal-Switching-qlcFiltrations}
	The following statements are equivalent:
	\begin{enumerate}
		\item $\mathbb{F}$ satisfies the usual conditions (right-continuous and $\mathsf{P}$-complete) and is quasi-left-continuous;
		\item For every bounded (and then for every uniformly integrable) c\`{a}dl\`{a}g martingale $M$ and every predictable time $S$, we have $\triangle_{S} M = 0$ a.s.;
		\item For every increasing sequence of stopping times $\lbrace S_{n} \rbrace$ with limit, $\lim_{n}S_{n} = S$, we have
		\[
		\mathcal{F}_{S} = \bigvee_{n} \mathcal{F}_{S_{n}}.
		\]
	\end{enumerate}
\end{proposition}
\subsection{Some notation.}\label{Section:Continuous-Time-Optimal-Switching-Notation}
Let us now define some notation that is frequently used below.
\begin{enumerate}
	\item For $1 \le p < \infty$, we write $L^{p}$ to denote the set of random variables $Z$ satisfying $\mathsf{E}\left[|Z|^{p}\right] < \infty$. 
	\item Let $\mathcal{Q}$ denote the set of adapted, c\`{a}dl\`{a}g processes which are quasi-left-continuous.
	\item Let $\mathcal{M}^{2}$ denote the set of progressively measurable processes $X = \left(X_{t}\right)_{0 \le t \le T}$ satisfying,
	\[
	\mathsf{E}\left[\int_{0}^{T}|X_{t}|^{2}{d}t\right] < \infty.
	\]
	\item Let $\mathcal{S}^{2}$ denote the set of adapted, c\`{a}dl\`{a}g processes $X$ satisfying:
	
	\[
	\mathsf{E}\left[\left(\sup\nolimits_{0 \le t \le T}\left|X_{t}\right|\right)^{2}\right] < \infty.
	\]
\end{enumerate}
\subsection{Properties of Snell envelopes.}\label{Section:Optimal-Switching-SnellEnvelopes}
The following properties of Snell envelopes are also essential for our results. Recall that a progressively measurable process $X$ is said to belong to class $[D]$ if the set of random variables $\lbrace X_{\tau}, \tau \in \mathcal{T} \rbrace$ is uniformly integrable.

\begin{proposition}\label{Proposition:Optimal-Switching-SnellEnvelopeProperties}
	Let $U = (U_{t})_{0 \le t \le T}$ be an adapted, $\mathbb{R}$-valued, c\`{a}dl\`{a}g process that belongs to class $[D]$. Then there exists a unique (up to indistinguishability), adapted $\mathbb{R}$-valued c\`{a}dl\`{a}g process $Z = (Z_{t})_{0 \le t \le T}$ such that $Z$ is the smallest supermartingale which dominates $U$. The process $Z$ is called the Snell envelope of $U$ and it enjoys the following properties.
	\begin{enumerate}
		\item For any stopping time $\theta$ we have:
		\begin{equation}
		Z_{\theta} = \esssup_{\tau \in \mathcal{T}_{\theta}}\mathsf{E}\left[U_{\tau} \vert \mathcal{F}_{\theta}\right],\text{ and therefore }Z_{T} = U_{T}.
		\end{equation}
		\item Meyer decomposition: There exist a uniformly integrable right-continuous martingale $M$ and two non-decreasing, adapted, predictable and integrable processes $A$ and $B$, with $A$ continuous and $B$ purely discontinuous, such that for all $0 \le t \le T$,
		\begin{equation}\label{SnellEnvelope:MeyerDecomposition}
		Z_{t} = M_{t} - A_{t} - B_{t}, \hspace{1em} A_{0} = B_{0} = 0.
		\end{equation}
		Furthermore, the jumps of $B$ satisfy $\lbrace \triangle B > 0 \rbrace \subset \lbrace Z_{-} = U_{-} \rbrace$.
		\item Let a stopping time $\theta$ be given and let $\{\tau_{n}\}_{n \ge 0}$ be an increasing sequence of stopping times tending to a limit $\tau$ such that each $\tau_{n} \in \mathcal{T}_{\theta}$ and satisfies $\mathsf{E}\left[U^{-}_{\tau_{n}}\right] < \infty$. Suppose the following condition is satisfied for any such sequence,
		\begin{equation}\label{eq:Optimal-Switching-LeftUpperSemicontinuity}
		\limsup_{n \to \infty} U_{\tau_{n}} \le U_{\tau}
		\end{equation}
		Then the stopping time $\tau^{*}_{\theta}$ defined by
		\begin{equation}\label{eq:Optimal-Switching-DebutTime}
		\tau^{*}_{\theta} = \inf\lbrace t \ge \theta \colon Z_{t} = U_{t} \rbrace \wedge T
		\end{equation}
		is optimal after $\theta$ in the sense that:
		\begin{equation}\label{eq:Optimal-Switching-OptimalityOfDebutTime}
		Z_{\theta} = \mathsf{E}\left[Z_{\tau^{*}_{\theta}} \vert \mathcal{F}_{\theta}\right] = \mathsf{E}\left[U_{\tau^{*}_{\theta}} \vert \mathcal{F}_{\theta}\right] = \esssup_{\tau \in \mathcal{T}_{\theta}}\mathsf{E}\left[U_{\tau} \vert \mathcal{F}_{\theta}\right].
		\end{equation}
		\item For every $\theta \in \mathcal{T}$, if $\tau^{*}_{\theta}$ is the stopping time defined in equation~\eqref{eq:Optimal-Switching-DebutTime}, then the stopped process $\left(Z_{t \wedge \tau^{*}_{\theta}}\right)_{\theta \le t \le T}$ is a (uniformly integrable) c\`{a}dl\`{a}g martingale.
		\item Let $\lbrace U^{n} \rbrace_{n \ge 0}$ and $U$ be adapted, c\`{a}dl\`{a}g and of class $[D]$ and let $Z^{U^{n}}$ and $Z$ denote the Snell envelopes of $U^{n}$ and $U$ respectively. If the sequence $\lbrace U^{n} \rbrace_{n \ge 0}$ is increasing and converges pointwise to $U$, then the sequence $\lbrace Z^{U^{n}} \rbrace_{n \ge 0}$ is also increasing and converges pointwise to $Z$. Furthermore, if $U \in \mathcal{S}^{2}$ then $Z \in \mathcal{S}^{2}$.
	\end{enumerate}
\end{proposition}
References for these properties can be found in the appendix of \cite{Hamadene2002} and other references such as \cite{ElKaroui1981,Morimoto1982,Peskir2006}. Proof of the fifth property can be found in Proposition 2 of \cite{Djehiche2009}. We also have the following result concerning integrability of the components in the Doob-Meyer decomposition.
\begin{proposition}\label{Proposition:Optimal-Switching-SquareIntegrableMartingale}
	For $0 \le t \le T$, let $Z_{t} = M_{t} - A_{t}$ where
	\begin{enumerate}
		\item the process $Z = (Z_{t})_{0 \le t \le T}$ is in $\mathcal{S}^{2}$;
		\item the process $M = (M_{t})_{0 \le t \le T}$ is a c\`{a}dl\`{a}g, quasi-left-continuous martingale with respect to $\mathbb{F}$;
		\item the process $A = (A_{t})_{0 \le t \le T}$ is an $\mathbb{F}$-adapted c\`{a}dl\`{a}g increasing process.
	\end{enumerate}
	Then $A$ (and therefore $M$) is also in $\mathcal{S}^{2}$.
\end{proposition}
\begin{proof}
The proof essentially uses an integration by parts formula on $(A_{T})^{2}$ and the decomposition $Z = M - A$ in the hypothesis. See Proposition~A.5 of \cite{Hamadene2002} for further details, noting that the same proof works for quasi-left-continuous $M$.
\end{proof}
\section{Assumptions}\label{Section:Optimal-Switching-Assumptions}

\begin{assumption}\label{Assumption:QLCFiltration}
	The filtration $\mathbb{F}$ satisfies the usual conditions of right-continuity and $\mathsf{P}$-completeness and is also quasi-left-continuous.
\end{assumption}

\begin{assumption}\label{Assumption:ProfitandCost}
	For every $i,j \in \mathbb{I}$ we suppose:
	\begin{enumerate}
		\item the instantaneous profit satisfies $\psi_{i} \in \mathcal{M}^{2}$;
		\item the switching cost satisfies $\gamma_{i,j} \in \mathcal{Q} \cap \mathcal{S}^{2}$.
		\item the terminal data $\Gamma_{i} \in L^{2}$ and is $\mathcal{F}_{T}$-measurable.
	\end{enumerate}
\end{assumption}

\begin{assumption}\label{Assumption:SwitchingCosts}
	For every $i,j,k \in \mathbb{I}$ and $\forall t \in [0,T]$, we have a.s.:
	\begin{equation}\label{eq:Optimal-Switching-NoArbitrage}
	\begin{cases}
	\gamma_{i,i}\left(t\right) = 0 \\
	\gamma_{i,k}\left(t\right) < \gamma_{i,j}\left(t\right) + \gamma_{j,k}\left(t\right),\hspace{1em} \text{ if } i \neq j \text{ and } j \neq k,\\
	\Gamma_{i} \ge \max\limits_{j \neq i}\left\lbrace \Gamma_{j} -\gamma_{i,j}(T) \right\rbrace.
	\end{cases}
	\end{equation}
\end{assumption}

\begin{remark}\label{Remark:SubOptimalSwitchTwice}
	The first line in condition~\eqref{eq:Optimal-Switching-NoArbitrage} shows there is no cost for staying in the same mode. The other two rule out possible arbitrage opportunities (also see \cite{Guo2008b,Hamadene2012}). In particular, we can always restrict our attention to those strategies $\alpha = \left(\tau_{n},\iota_{n}\right)_{n \ge 0} \in \mathcal{A}$ such that $\mathsf{P}\big( \lbrace \tau_{n} = \tau_{n+1}, \tau_{n} < T \rbrace) = 0$ for $n \ge 1$. Indeed, if $H_{n} \coloneqq \lbrace \tau_{n} = \tau_{n+1}, \tau_{n} < T \rbrace$ satisfies $\mathsf{P}(H_{n}) > 0$ for some $n \ge 1$, then by the second line in condition~\eqref{eq:Optimal-Switching-NoArbitrage} we get
	\begin{align*}
	\left(\gamma_{\iota_{n-1},\iota_{n}}(\tau_{n}) + \gamma_{\iota_{n},\iota_{n+1}}(\tau_{n+1})\right)\mathbf{1}_{H_{n}} & = \left(\gamma_{\iota_{n-1},\iota_{n}}(\tau_{n}) + \gamma_{\iota_{n},\iota_{n+1}}(\tau_{n})\right)\mathbf{1}_{H_{n}} \\
	& > \left(\gamma_{\iota_{n-1},\iota_{n+1}}(\tau_{n})\right)\mathbf{1}_{H_{n}}
	\end{align*}
	which shows it is suboptimal to switch twice at the same time.
\end{remark}
\section{A verification theorem.}\label{Section:Optimal-Switching-VerificationTheorem}
Throughout this section, we suppose that there exist processes $Y^{1},\ldots,Y^{m}$ in $\mathcal{Q} \cap \mathcal{S}^{2}$ defined by
\begin{equation}\label{eq:Optimal-Switching-OptimalSwitchingProcessesDefinition}
\begin{split}
Y^{i}_{t} & = \esssup\limits_{\tau \in \mathcal{T}_{t}}\mathsf{E}\left[\int_{t}^{\tau}\psi_{i}(s){d}s + \Gamma_{i}\mathbf{1}_{\{\tau = T \}} + \max\limits_{j \neq i}\left\lbrace Y^{j}_{\tau} -\gamma_{i,j}(\tau) \right\rbrace\mathbf{1}_{\{\tau < T \}}\biggm\vert \mathcal{F}_{t}\right],\\
Y^{i}_{T} & = \Gamma_{i}.
\end{split}
\end{equation}
Sufficient conditions ensuring the existence of $Y^{1},\ldots,Y^{m}$ with these properties are given in Section~\ref{Section:Optimal-Switching-Existence-Proof}. Theorem~\ref{Theorem:Optimal-Switching-VerificationPartial} below verifies that the solution to the optimal switching problem~\eqref{eq:Optimal-Switching-ValueFunction} can be written in terms of these $m$ stochastic processes. In preparation of this \emph{verification theorem}, we need a few preliminary results. Let $U^{i} = \left(U^{i}_{t}\right)_{0 \le t \le T}$, $i \in \mathbb{I}$, be a c\`{a}dl\`{a}g process defined by:
\begin{equation}\label{eq:Optimal-Switching-InterconnectedObstacleProcess}
U^{i}_{t} \coloneqq \Gamma_{i}\mathbf{1}_{\{ t = T \}} + \max\limits_{j \neq i} \left\lbrace Y^{j}_{t} - \gamma_{i,j}\left(t\right)\right\rbrace\mathbf{1}_{\{ t < T \}}, \hspace{1em} 0 \le t \le T.
\end{equation}
Recall that for every $i,j \in \mathbb{I}$ we have $\gamma_{i,j}, Y^{i} \in \mathcal{Q} \cap \mathcal{S}^{2}$, $\Gamma_{i} \in L^{2}$ by assumption. Hence the process $U^{i} \in \mathcal{S}^{2}$ and is therefore of class $[D]$. Recalling Proposition~\ref{Proposition:Optimal-Switching-SnellEnvelopeProperties} and rewriting equation~\eqref{eq:Optimal-Switching-OptimalSwitchingProcessesDefinition} for $Y^{i}_{t}$ as follows,
\begin{align*}\label{eq:Optimal-Switching-SnellEnvelopeCharacterisation}
Y^{i}_{t} & = \esssup\limits_{\tau \in \mathcal{T}_{t}}\mathsf{E}\left[\int_{t}^{\tau}\psi_{i}(s){d}s + U^{i}_{\tau}\biggm\vert \mathcal{F}_{t}\right] \nonumber \\
& = \esssup\limits_{\tau \in \mathcal{T}_{t}}\mathsf{E}\left[\int_{0}^{\tau}\psi_{i}(s){d}s + U^{i}_{\tau}\biggm\vert \mathcal{F}_{t}\right] - \int_{0}^{t}\psi_{i}(s){d}s,\hspace{1em}\mathsf{P}-\text{a.s.}
\end{align*}
we can verify that $\left(Y^{i}_{t} + \int_{0}^{t} \psi_{i}(s){d}s \right)_{0 \le t \le T}$ is the Snell envelope of $\left(U^{i}_{t} + 
\int_{0}^{t} \psi_{i}(s){d}s\right)_{0 \le t \le T}$.

\begin{lemma}\label{Lemma:Optimal-Switching-VerificationLemma}
	Suppose that $Y^{1},\ldots,Y^{m}$ defined in \eqref{eq:Optimal-Switching-OptimalSwitchingProcessesDefinition} are in $\mathcal{Q} \cap \mathcal{S}^{2}$. For each $i \in \mathbb{I}$, let $U^{i}$ be defined as in equation~\eqref{eq:Optimal-Switching-InterconnectedObstacleProcess}. Then for every $\tau_{n} \in \mathcal{T}$ and $\mathcal{F}_{\tau_{n}}$-measurable $\iota_{n} \colon \Omega \to \mathbb{I}$, we have
	\begin{equation}\label{eq:Optimal-Switching-VerificationLemmaSnellEnvelope} 
	Y^{\iota_{n}}_{t} = \esssup\limits_{\tau \in \mathcal{T}_{t}}\mathsf{E}\left[\int_{t}^{\tau}\psi_{\iota_{n}}(s){d}s + U^{\iota_{n}}_{\tau} \biggm \vert \mathcal{F}_{t}\right], \hspace{1em} \mathsf{P}-\text{a.s. }\forall \hspace{2bp} \tau_{n} \le t  \le T.
	\end{equation}
	Furthermore, there exist a uniformly integrable c\`{a}dl\`{a}g martingale $M^{\iota_{n}} = \left(M^{\iota_{n}}_{t}\right)_{\tau_{n} \le t \le T}$ and a predictable, continuous, increasing process $A^{\iota_{n}} = \left(A^{\iota_{n}}_{t}\right)_{\tau_{n} \le t \le T}$ such that
	\begin{equation}\label{eq:Optimal-Switching-VerificationLemmaMeyerDecomposition}
	Y^{\iota_{n}}_{t} + \int_{0}^{t}\psi_{\iota_{n}}(s){d}s = M^{\iota_{n}}_{t} - A^{\iota_{n}}_{t},\hspace{1em}\mathsf{P}-\text{a.s. }\forall \hspace{2bp} \tau_{n} \le t \le T.
	\end{equation}
\end{lemma}
\begin{proof}
The claim~\eqref{eq:Optimal-Switching-VerificationLemmaSnellEnvelope} is established in the same way as the first few lines of Theorem 1 in \cite{Djehiche2009} so the proof is sketched. We need to show that $Y^{\iota_{n}}_{t} + \int_{0}^{t}\psi_{\iota_{n}}(s){d}s$ is the Snell envelope of $U^{\iota_{n}}_{t} + \int_{0}^{t}\psi_{\iota_{n}}(s){d}s$ for $\tau_{n} \le t \le T$. Our previous discussion established under the current hypotheses that, for every $i \in \mathbb{I}$, $Y^{i}_{t} + \int_{0}^{t}\psi_{i}(s){d}s$ is the Snell envelope of $U^{i}_{t} + \int_{0}^{t}\psi_{i}(s){d}s$ on $[0,T]$. Since $\mathbf{1}_{\lbrace \iota_{n} = i \rbrace}$ is non-negative and $\mathcal{F}_{\tau_{n}}$-measurable, we can show that $\left(Y^{i}_{t} + \int_{0}^{t}\psi_{i}(s){d}s\right)\mathbf{1}_{\lbrace \iota_{n} = i \rbrace}$ is the smallest c\`{a}dl\`{a}g supermartingale dominating $\left(U^{i}_{t} + \int_{0}^{t}\psi_{i}(s){d}s\right)\mathbf{1}_{\lbrace \iota_{n} = i \rbrace}$ on $[\tau_{n},T]$. By summing over $i \in \mathbb{I}$ (recall $\mathbb{I}$ is finite), we have $\left(Y^{\iota_{n}}_{t} + \int_{0}^{t}\psi_{\iota_{n}}(s){d}s\right)$ is the smallest c\`{a}dl\`{a}g supermartingale dominating $\left(U^{\iota_{n}}_{t} + \int_{0}^{t}\psi_{\iota_{n}}(s){d}s\right)$ for $\tau_{n} \le t \le T$. In particular, we have
\[
Y^{\iota_{n}}_{t} + \int_{0}^{t}\psi_{\iota_{n}}(s){d}s = \esssup\limits_{\tau \in \mathcal{T}_{t}}\mathsf{E}\left[\int_{0}^{\tau}\psi_{\iota_{n}}(s){d}s + U^{\iota_{n}}_{\tau}\biggm\vert \mathcal{F}_{t}\right],\hspace{1em}\mathsf{P}-\text{a.s. }\forall \hspace{1bp} t \le \tau_{n} \le T,
\]
and equation~\eqref{eq:Optimal-Switching-VerificationLemmaSnellEnvelope} follows by $\mathcal{F}_{t}$-measurability of the integral term for $t \ge \tau_{n}$.

For the second part of the claim, we use the unique Meyer decomposition of the Snell envelope (property 2 of Proposition~\ref{Proposition:Optimal-Switching-SnellEnvelopeProperties}) to show that for every $i \in \mathbb{I}$,
\[
Y^{i}_{t} + \int_{0}^{t}\psi_{i}(s){d}s = M^{i}_{t} - A^{i}_{t}\hspace{1em} \text{for } t \in [0,T],
\]
where $M^{i} = \left(M^{i}_{t}\right)_{0 \le t \le T}$ is a c\`{a}dl\`{a}g uniformly integrable martingale and $A^{i} = \left(A^{i}_{t}\right)_{0 \le t \le T}$ is a predictable, increasing process. The Snell envelope $\left(Y^{i}_{t} + \int_{0}^{t}\psi_{i}(s){d}s\right)_{0 \le t \le T}$ is in $\mathcal{Q} \cap \mathcal{S}^{2}$ since $Y^{i} \in \mathcal{Q} \cap \mathcal{S}^{2}$ and $\psi_{i} \in \mathcal{M}^{2}$. This means the Snell envelope is a regular supermartingale of class $[D]$ and Theorem \MakeUppercase{\romannumeral 7}.10 of \cite{Dellacherie1982} asserts that its compensator, $A^{i}$, is continuous.

Using the Meyer decomposition, we see that
\begin{equation}\label{eq:Optimal-Switching-VerificationMeyerDecompositionFull}
Y^{\iota_{n}}_{t} + \int_{0}^{t}\psi_{\iota_{n}}(s){d}s \coloneqq \sum_{i \in \mathbb{I}}\left(Y^{i}_{t} + \int_{0}^{t}\psi_{i}(s){d}s\right)\mathbf{1}_{\{\iota_{n} = i\}} = \sum_{i \in \mathbb{I}}\left(M^{i}_{t} - A^{i}_{t}\right)\mathbf{1}_{\{\iota_{n} = i\}}.
\end{equation}
Now, using $\mathbf{1}_{\lbrace \iota_{n} = i \rbrace}$ is non-negative and $\mathcal{F}_{t}$-measurable for $t \ge \tau_{n}$, we see that $M^{\iota_{n}}$ defined on $[\tau_{n},T]$ by 
\begin{equation}\label{eq:Optimal-Switching-VerificationMeyerDecompositionMartingale}
M^{\iota_{n}(\omega)}_{t}(\omega) \coloneqq \sum_{i \in \mathbb{I}}M^{i}_{t}(\omega)\mathbf{1}_{\{\iota_{n} = i\}}(\omega),\hspace{1em} \forall\hspace{1bp} (\omega,t) \in [\tau_{n},T]
\end{equation}
is a uniformly integrable c\`{a}dl\`{a}g martingale $\mathsf{P}$-a.s. for every $\tau_{n} \le t \le T$. Likewise, $A^{\iota_{n}}$ defined on $[\tau_{n},T]$ by
\begin{equation}\label{eq:Optimal-Switching-VerificationMeyerDecompositionCompensator}
A^{\iota_{n}(\omega)}_{t}(\omega) \coloneqq \sum_{i \in \mathbb{I}}A^{i}_{t}(\omega)\mathbf{1}_{\{\iota_{n} = i\}}(\omega),\hspace{1em} \forall\hspace{1bp} (\omega,t) \in [\tau_{n},T]
\end{equation}
is a continuous, predictable increasing process $\mathsf{P}$-a.s. for every $\tau_{n} \le t \le T$. By equation~\eqref{eq:Optimal-Switching-VerificationMeyerDecompositionFull}, $M^{\iota_{n}}_{t}$ and $A^{\iota_{n}}_{t}$ provide the (unique) Meyer decomposition of $Y^{\iota_{n}}_{t}$ $\mathsf{P}$-a.s. for every $\tau_{n} \le t \le T$.
\end{proof}

\begin{theorem}[Verification]\label{Theorem:Optimal-Switching-VerificationPartial}
	Suppose there exist $m$ unique processes $Y^{1},\ldots,Y^{m}$ in $\mathcal{Q} \cap \mathcal{S}^{2}$ which satisfy equation~\eqref{eq:Optimal-Switching-OptimalSwitchingProcessesDefinition}. Define a sequence of times $\left\lbrace\tau^{*}_{n}\right\rbrace_{n \ge 0}$ and mode indicators $\left\lbrace\iota^{*}_{n}\right\rbrace_{n \ge 0}$ as follows:
	\begin{gather}\label{eq:Optimal-Switching-OptimalStoppingStrategy}
	\tau^{*}_{0} = t,\hspace{1em} \iota^{*}_{0} = i, \\
	\begin{cases}
	\tau^{*}_{n} = \inf\left\lbrace s \ge \tau^{*}_{n-1} \colon Y^{\iota^{*}_{n-1}}_{s} = \max\limits_{j \neq \iota^{*}_{n-1}} \left(Y^{j}_{s} - \gamma_{\iota^{*}_{n-1},j}\left(s\right) \right)  \right\rbrace \wedge T,\\
	\iota^{*}_{n} = \sum\limits_{j \in \mathbb{I}} j \mathbf{1}_{F^{\iota^{*}_{n-1}}_{j}}
	\end{cases} \nonumber \\
	\text{for } n \ge 1, \text{ where } F^{\iota^{*}_{n-1}}_{j} \text{ is the event} : \nonumber \\
	F^{\iota^{*}_{n-1}}_{j} \coloneqq \left\lbrace Y^{j}_{\tau^{*}_{n}} -\gamma_{\iota^{*}_{n-1},j}\left(\tau^{*}_{n}\right) = \max\limits_{k \neq \iota^{*}_{n-1}}\left( Y^{k}_{\tau^{*}_{n}} -\gamma_{\iota^{*}_{n-1},k}\left(\tau^{*}_{n}\right) \right) \right\rbrace. \nonumber
	\end{gather}
	Then the sequence $\alpha^{*} = \left(\tau^{*}_{n},\iota^{*}_{n}\right)_{n \ge 0} \in \mathcal{A}_{t,i}$ and satisfies \begin{equation}\label{eq:Optimal-Switching-UpperBoundOnThePerformanceIndex}
	Y^{i}_{t} = J(\alpha^{*};t,i) = \esssup\limits_{\alpha \in \mathcal{A}_{t,i}}J(\alpha;t,i) \hspace{1em} \mathsf{P}-\text{a.s.}
	\end{equation}
\end{theorem}
\begin{proof}
Standard arguments can be used to verify that $\tau^{*}_{n}$ is a stopping time and each $\iota^{*}_{n}$ is $\mathcal{F}_{\tau^{*}_{n}}$-measurable. The appendix confirms that $\alpha^{*} \in \mathcal{A}_{t,i}$. As for the claim \eqref{eq:Optimal-Switching-UpperBoundOnThePerformanceIndex}, it holds trivially for $t = T$ since $Y^{i}_{T} = \Gamma_{i} = V(t,i)$ a.s. for every $i \in \mathbb{I}$. Henceforth, we assume that $t \in [0,T)$.

Recall the process $U^{i} = \left(U^{i}_{t}\right)_{0 \le t \le T}$ defined in equation~\eqref{eq:Optimal-Switching-InterconnectedObstacleProcess}. By our assumptions on $Y^{i},\psi_{i},\Gamma_{i}$ and $\gamma_{i,j}$ for every $i,j \in \mathbb{I}$, we have $U^{i} \in \mathcal{S}^{2}$ and we assert that $\left(Y^{i}_{t} + \int_{0}^{t} \psi_{i}(s){d}s \right)_{0 \le t \le T}$ is the Snell envelope of $\left(U^{i}_{t} + \int_{0}^{t} \psi_{i}(s){d}s \right)_{0 \le t \le T}$. For $i,j \in \mathbb{I}$, using $Y^{j}_{T} = \Gamma_{j}$ $\mathsf{P}$-a.s., quasi-left-continuity of $Y^{j}$ and $\gamma_{i,j}$, and Assumption~\ref{Assumption:SwitchingCosts} on the terminal condition for the switching costs, we have
\[ 
\lim_{t \uparrow T}\left(\max_{j \neq i}\left\lbrace Y^{j}_{t} -\gamma_{i,j}(t) \right\rbrace\right) = \max_{j \neq i}\left\lbrace \Gamma_{j} - \gamma_{i,j}(T) \right\rbrace \le \Gamma_{i} \hspace{1em} \mathsf{P}-\text{a.s.}
\]
Therefore, $U^{i}$ is quasi-left-continuous on $[0,T)$ and $\lim_{t \uparrow T}U^{i}_{t} \le U^{i}_{T}$ $\mathsf{P}$-a.s. Combining this with the continuity of the integral, we see that $\left(U^{i}_{t} + \int_{0}^{t} \psi_{i}(s){d}s \right)_{0 \le t \le T}$ satisfies the hypotheses of property 3 in Proposition~\ref{Proposition:Optimal-Switching-SnellEnvelopeProperties}. Let $\left(\tau_{n}^{*},\iota_{n}^{*}\right)_{n \ge 0}$ be the pair of random times and mode indicators in the statement of the theorem and $\mathbf{u}^{*}$ be the associated mode indicator function. In conjunction with Lemma~\ref{Lemma:Optimal-Switching-VerificationLemma}, $\lbrace \tau_{n}^{*}\rbrace$ defines a sequence of stopping times where, for $n \ge 1$, $\tau^{*}_{n}$ is optimal for an appropriately defined optimal stopping problem.

The remaining arguments, which are similar to those establishing Theorem~1 in \cite{Djehiche2009}, are only sketched here. The main idea is as follows: starting from an initial mode $i \in \mathbb{I}$ at time $t \in [0,T]$, iteratively solve the optimal stopping problem on the right-hand-side of \eqref{eq:Optimal-Switching-OptimalSwitchingProcessesDefinition} using the theory of Snell envelopes and Lemma~\ref{Lemma:Optimal-Switching-VerificationLemma}. The minimal optimal stopping times characterise the switching times whilst the maximising modes are paired with them to give the switching strategy. This characterisation will eventually lead to:
\begin{equation}\label{eq:Optimal-Switching-VerificationProof5}
\begin{split}
\forall N \ge 1,\hspace{1em} Y^{i}_{t} = {} & \mathsf{E}\left[\int_{t}^{\tau^{*}_{N}}\psi_{\mathbf{u}^{*}_{s}}(s){d}s + \sum_{n = 1}^{N}\Gamma_{\iota^{*}_{n-1}}\mathbf{1}_{\lbrace \tau^{*}_{n-1} < T \rbrace}\mathbf{1}_{\lbrace \tau^{*}_{n} = T \rbrace} - \sum_{n = 1}^{N}\gamma_{\iota^{*}_{n-1},\iota^{*}_{n}}\left(\tau^{*}_{n}\right)\mathbf{1}_{\lbrace \tau^{*}_{n} < T \rbrace} \biggm \vert \mathcal{F}_{t} \right] \\
& + \mathsf{E}\left[Y^{\iota^{*}_{N}}_{\tau^{*}_{N}} \mathbf{1}_{\lbrace \tau^{*}_{N} < T \rbrace} \biggm \vert \mathcal{F}_{t} \right]
\end{split}
\end{equation}

By Lemma~\ref{Lemma:Optimal-Switching-FiniteStrategyCandidateOptimal} and Theorem~\ref{Theorem:Optimal-Switching-SquareIntegrableSwitchingCosts} in the appendix respectively, the times $\left\lbrace\tau_{n}^{*}\right\rbrace_{n \ge 0}$ satisfy the finiteness condition~\eqref{eq:Optimal-Switching-FiniteStrategy} and $\mathsf{E}\big[\sup_{n}\big|C^{\alpha^{*}}_{n}\big|\big] < \infty$ holds for the cumulative switching costs. Appealing also to the conditional dominated convergence theorem (cf. \eqref{eq:Optimal-Switching-ConvergenceOfSwitchingCost}), we may take the limit as $N \to \infty$ in equation~\eqref{eq:Optimal-Switching-VerificationProof5} and use the definition of $\mathbf{u}^{*}$ to get:
\begin{equation*}\label{eq:Optimal-Switching-VerificationProof6}
Y^{i}_{t} = \mathsf{E}\left[\int_{t}^{T}\psi_{\mathbf{u}^{*}_{s}}(s){d}s + \Gamma_{\mathbf{u}^{*}_{T}} - \sum_{n \ge 1}\gamma_{\iota^{*}_{n-1},\iota^{*}_{n}}\left(\tau^{*}_{n}\right)\mathbf{1}_{\lbrace \tau^{*}_{n} < T \rbrace} \biggm \vert \mathcal{F}_{t}\right] = J(\alpha^{*};t,i).
\end{equation*}
Now, take any arbitrary admissible strategy $\alpha = \left(\tau_{n},\iota_{n}\right)_{n \ge 0} \in \mathcal{A}_{t,i}$. Since the sequence $(\tau_{n},\iota_{n})_{n \ge 1}$, does not necessarily achieve the essential suprema / maxima in the iterated optimal stopping problems, we have for all $N \ge 1$:
\[
\begin{split}
Y^{i}_{t} \ge {} & \mathsf{E}\left[\int_{t}^{\tau_{N}}\psi_{\mathbf{u}_{s}}(s){d}s + \sum_{n = 1}^{N}\Gamma_{\iota_{n-1}}\mathbf{1}_{\lbrace \tau_{n-1} < T \rbrace}\mathbf{1}_{\lbrace \tau_{n} = T \rbrace} - \sum_{n = 1}^{N}\gamma_{\iota_{n-1},\iota_{n}}\left(\tau_{n}\right)\mathbf{1}_{\lbrace \tau_{n} < T \rbrace} \biggm \vert \mathcal{F}_{t} \right] \\
& + \mathsf{E}\left[Y^{\iota_{N}}_{\tau_{N}} \mathbf{1}_{\lbrace \tau_{N} < T \rbrace} \biggm \vert \mathcal{F}_{t} \right]
\end{split}
\]
Passing to the limit $N \to \infty$ and using the conditional dominated convergence theorem, we obtain
\[
J(\alpha^{*};t,i) = Y^{i}_{t} \ge \mathsf{E}\left[\int_{t}^{T}\psi_{\mathbf{u}_{s}}(s){d}s + \Gamma_{\mathbf{u}_{T}} - \sum_{n \ge 1}\gamma_{\iota_{n-1},\iota_{n}}\left(\tau_{n}\right)\mathbf{1}_{\lbrace \tau_{n} < T \rbrace} \biggm \vert \mathcal{F}_{t}\right] = J(\alpha;t,i).
\]
Since $\alpha \in \mathcal{A}_{t,i}$ was arbitrary we have just proved \eqref{eq:Optimal-Switching-UpperBoundOnThePerformanceIndex}.
\end{proof}

\section{Existence of the candidate optimal processes.}\label{Section:Optimal-Switching-Existence-Proof}
The existence of the processes $Y^{1},\ldots,Y^{m}$ which satisfy Theorem~\ref{Theorem:Optimal-Switching-VerificationPartial} is proved in this section following the arguments of \cite{Djehiche2009}. The interested reader may also compare the proof to that of Lemma 2.1 and Corollary 2.1 in \cite{Bayraktar2010}.
\subsection{The case of at most $n \ge 0$ switches.}
For each $n \ge 0$, define process $Y^{1,n},\ldots,Y^{m,n}$ recursively as follows: for $i \in \mathbb{I}$ and for any $0 \le t \le T$, first set
\begin{equation}\label{eq:Optimal-Switching-OptimalProcess0Switches}
Y^{i,0}_{t} = \mathsf{E}\left[\int_{t}^{T}\psi_{i}(s){d}s + \Gamma_{i} \biggm\vert \mathcal{F}_{t}\right],
\end{equation}
and for $n \ge 1$,
\begin{equation}\label{eq:Optimal-Switching-OptimalProcessnSwitches}
Y^{i,n}_{t} = \esssup\limits_{\tau \in \mathcal{T}_{t}}\mathsf{E}\left[\int_{t}^{\tau}\psi_{i}(s){d}s + \Gamma_{i}\mathbf{1}_{\{\tau = T \}} + \max\limits_{j \neq i}\left\lbrace Y^{j,n-1}_{\tau} - \gamma_{i,j}(\tau) \right\rbrace\mathbf{1}_{\{\tau < T \}}\biggm\vert \mathcal{F}_{t}\right].
\end{equation}

Define another process $\hat{U}^{i,n} = (\hat{U}^{i,n}_{t})_{0 \le t \le T}$ by:
\begin{equation*}
\hat{U}^{i,n}_{t} \coloneqq \int_{0}^{t}\psi_{i}(s){d}s + \Gamma_{i}\mathbf{1}_{\{ t = T \}} + \max_{j \neq i}\left\lbrace Y^{j,n-1}_{t} - \gamma_{i,j}(t) \right\rbrace\mathbf{1}_{\{ t < T \}}
\end{equation*}

If $\hat{U}^{i,n}$ is of class $[D]$, then by Proposition~\ref{Proposition:Optimal-Switching-SnellEnvelopeProperties} its Snell envelope exists and is defined by
\[
\esssup\limits_{\tau \in \mathcal{T}_{t}}\mathsf{E}\bigl[\hat{U}^{i,n}_{\tau}\big \vert \mathcal{F}_{t}\bigr]
= Y^{i,n}_{t} + \int_{0}^{t} \psi_{i}(s){d}s.
\]
Some properties of $Y^{i,n}$ which verify this are proved in the following lemma. In order to simplify some expressions in the proof, introduce a new process $\hat{Y}^{i,n} = (\hat{Y}^{i,n}_{t})_{0 \le t \le T}$ which is defined by:
\begin{equation*}\label{eq:Optimal-Switching-YHatProcess}
\hat{Y}^{i,n}_{t} \coloneqq Y^{i,n}_{t} + \int_{0}^{t}\psi_{i}(s){d}s.
\end{equation*}

\begin{lemma}\label{Lemma:Optimal-Switching-FinitelyManySwitchesQLC_S2}
	For all $n \ge 0$, the processes $Y^{1,n},\ldots,Y^{m,n}$ defined by \eqref{eq:Optimal-Switching-OptimalProcess0Switches} and \eqref{eq:Optimal-Switching-OptimalProcessnSwitches} are in $\mathcal{Q} \cap \mathcal{S}^{2}$.
\end{lemma}
\begin{proof}
The proof is similar to the one in \cite{Djehiche2009}. By $\mathcal{F}_{t}$-measurability of the integral term, we have
\[
\hat{Y}^{i,0}_{t} \coloneqq Y^{i,0}_{t} + \int_{0}^{t}\psi_{i}(s){d}s  = \mathsf{E}\left[\int_{0}^{T}\psi_{i}(s){d}s + \Gamma_{i} \biggm\vert \mathcal{F}_{t}\right].
\]
Since $\psi_{i} \in \mathcal{M}^{2}$ and $\Gamma_{i} \in L^{2}$, the conditional expectation is well-defined and $\hat{Y}^{i,0}$ is a uniformly integrable martingale which we can take to be c\`{a}dl\`{a}g (Section \MakeUppercase{\romannumeral 2}.67 of \cite{Rogers2000a}). By Doob's maximal inequality it follows that $\hat{Y}^{i,0} \in \mathcal{S}^{2}$ and therefore $Y^{i,0}$. Since the filtration is assumed to be quasi-left-continuous, Proposition~\ref{Proposition:Optimal-Switching-qlcFiltrations} verifies that $\hat{Y}^{i,0} \in \mathcal{Q}$ and therefore $Y^{i,0} \in \mathcal{Q}$. Therefore, $Y^{i,n} \in \mathcal{Q} \cap \mathcal{S}^{2}$ for every $i \in \mathbb{I}$ when $n = 0$.

Now, suppose by an induction hypothesis on $n \ge 0$ that for all $i \in \mathbb{I}$, $Y^{i,n} \in \mathcal{Q} \cap \mathcal{S}^{2}$. We first show that $Y^{i,n+1} \in \mathcal{S}^{2}$. By the induction hypothesis on $Y^{i,n}$ and since $\gamma_{i,j} \in \mathcal{Q} \cap \mathcal{S}^{2}$ and $\psi_{i} \in \mathcal{M}^{2}$, we verify that $\hat{U}^{i,n+1} \in \mathcal{S}^{2}$. Therefore, by Proposition~\ref{Proposition:Optimal-Switching-SnellEnvelopeProperties}, $\hat{Y}^{i,n+1}$ is the Snell envelope of $\hat{U}^{i,n+1}$. It is then not difficult to show that $\hat{U}^{i,n+1} \in \mathcal{S}^{2} \implies \hat{Y}^{i,n+1} \in \mathcal{S}^{2}$ (also property 5 of Proposition~\ref{Proposition:Optimal-Switching-SnellEnvelopeProperties}). Since $\psi_{i} \in \mathcal{M}^{2}$ we conclude that $Y^{i,n+1} \in \mathcal{S}^{2}$.

We now show that $Y^{i,n+1} \in \mathcal{Q}$ by arguing similarly as in the proof of Proposition 1.4a in \cite{Hamadene2003}. First, recall that $Y^{i,n}$ is in $\mathcal{Q} \cap \mathcal{S}^{2}$ for every $i \in \mathbb{I}$ by the induction hypothesis, and $\gamma_{ij} \in \mathcal{Q} \cap \mathcal{S}^{2}$ for every $i,j \in \mathbb{I}$. This means the process $\left(\max\nolimits_{j \neq i}\left\lbrace -\gamma_{i,j}(t) + Y^{j,n}_{t} \right\rbrace\right)_{0 \le t \le T}$ is also in $\mathcal{Q} \cap \mathcal{S}^{2}$. Using $Y^{j,n}_{T} = \Gamma_{j}$, $\mathsf{P}$-a.s. and Assumption~\ref{Assumption:SwitchingCosts} on the switching costs, we also have
\[ 
\lim_{t \uparrow T}\left(\max_{j \neq i}\left\lbrace Y^{j,n}_{t} - \gamma_{i,j}(t) \right\rbrace\right) = \max_{j \neq i}\left\lbrace \Gamma_{j} -\gamma_{i,j}(T) \right\rbrace \le \Gamma_{i}.
\]
Thus $\hat{U}^{i,n+1}$ is quasi-left-continuous on $[0,T)$ and has a possible positive jump at time $T$.

Next, by Proposition~\ref{Proposition:Optimal-Switching-SnellEnvelopeProperties}, $\hat{Y}^{i,n+1}$ has a unique Meyer decomposition:
\[
\hat{Y}^{i,n+1} = M - A - B,
\]
where $M$ is a right-continuous, uniformly integrable martingale, and $A$ and $B$ are predictable, non-decreasing processes which are continuous and purely discontinuous respectively. Let $\tau \in \mathcal{T}$ be any predictable time. The process $A$ is continuous so $A_{\tau^{-}} = A_{\tau}$ holds almost surely. Moreover, the martingale $M$ also satisfies $M_{\tau} = M_{\tau^{-}}$ a.s. since, by Proposition~\ref{Proposition:Optimal-Switching-qlcFiltrations}, it is quasi-left-continuous. Predictable jumps in $\hat{Y}^{i,n+1}$ therefore come from $B$, and we need only consider the two events $\lbrace \triangle_{\tau}B > 0 \rbrace$ and $\lbrace \triangle_{\tau}B = 0 \rbrace$ since $B$ is non-decreasing.

By property 2 of Proposition~\ref{Proposition:Optimal-Switching-SnellEnvelopeProperties}, we have
\[
\lbrace \triangle_{\tau}B > 0 \rbrace \subset \lbrace \hat{Y}^{i,n+1}_{\tau^{-}} = \hat{U}^{i,n+1}_{\tau^{-}} \rbrace
\]
and, using the dominating property of $\hat{Y}^{i,n+1}$ and non-negativity of the predictable jumps of $\hat{U}^{i,n+1}$, this gives
\begin{equation}\label{eq:Optimal-Switching-JumpAtPredictableTime1}
\mathsf{E}\left[\hat{Y}^{i,n+1}_{\tau^{-}}\mathbf{1}_{\lbrace \triangle_{\tau}B > 0 \rbrace}\right] = \mathsf{E}\left[\hat{U}^{i,n+1}_{\tau^{-}}\mathbf{1}_{\lbrace \triangle_{\tau}B > 0 \rbrace}\right] \le \mathsf{E}\left[\hat{U}^{i,n+1}_{\tau}\mathbf{1}_{\lbrace \triangle_{\tau}B > 0 \rbrace}\right] \le \mathsf{E}\left[\hat{Y}^{i,n+1}_{\tau}\mathbf{1}_{\lbrace \triangle_{\tau}B > 0 \rbrace}\right]
\end{equation}

On the other hand, the Meyer decomposition of $\hat{Y}^{i,n+1}$ and the almost sure continuity of $M$ and $A$ at $\tau$ yield the following:
\begin{align}\label{eq:Optimal-Switching-JumpAtPredictableTime2}
\mathsf{E}\left[\hat{Y}^{i,n+1}_{\tau^{-}}\mathbf{1}_{\lbrace \triangle_{\tau}B = 0 \rbrace}\right] & = \mathsf{E}\left[\left(M_{\tau^{-}} - A_{\tau^{-}} - B_{\tau^{-}}\right)\mathbf{1}_{\lbrace \triangle_{\tau}B = 0 \rbrace}\right] \nonumber \\
& = \mathsf{E}\left[\left(M_{\tau} - A_{\tau} - B_{\tau}\right)\mathbf{1}_{\lbrace \triangle_{\tau}B = 0 \rbrace}\right] \nonumber \\
& = \mathsf{E}\left[\hat{Y}^{i,n+1}_{\tau}\mathbf{1}_{\lbrace \triangle_{\tau}B = 0 \rbrace}\right]
\end{align}

From \eqref{eq:Optimal-Switching-JumpAtPredictableTime1} and \eqref{eq:Optimal-Switching-JumpAtPredictableTime2} we get the inequality, $\mathsf{E}\big[\hat{Y}^{i,n+1}_{\tau^{-}}\big] \le \mathsf{E}\big[\hat{Y}^{i,n+1}_{\tau}\big]$. However, $\mathsf{E}\big[\hat{Y}^{i,n+1}_{\tau^{-}}\big] \ge \mathsf{E}\big[\hat{Y}^{i,n+1}_{\tau}\big]$ since $\hat{Y}^{i,n+1}$ is a right-continuous supermartingale (in $\mathcal{S}^{2}$) and $\tau$ is predictable (Theorem \MakeUppercase{\romannumeral 6}.14 of \cite{Dellacherie1982}). Thus $\mathsf{E}\big[\hat{Y}^{i,n+1}_{\tau^{-}}\big] = \mathsf{E}\big[\hat{Y}^{i,n+1}_{\tau}\big]$ for every predictable time $\tau$. This means $Y^{i,n+1}$ is a regular supermartingale (of class $[D]$) and, by Theorem \MakeUppercase{\romannumeral 7}.10 of \cite{Dellacherie1982}, the predictable non-decreasing component of the Meyer decomposition of $Y^{i,n+1}$ must be continuous. Therefore, $B \equiv 0$ and $Y^{i,n+1} \in \mathcal{Q}$ since the only jumps it experiences are those from the quasi-left-continuous martingale $M$.
 \end{proof}
\begin{lemma}\label{Lemma:OptimalSwitching-ConvergenceOfTheOptimalProcesses}
	For every $i \in \mathbb{I}$, the process $Y^{i,n}$ solves the optimal switching problem with at most $n \ge 0$ switches:
	\begin{equation}\label{eq:Optimal-Switching-AtMostnSwitches}
	Y^{i,n}_{t} = \esssup\limits_{\alpha \in \mathcal{A}^{n}_{t,i}}\mathsf{E}\left[\int_{t}^{T}\psi_{\mathbf{u}_{s}}(s){d}s + \Gamma_{\mathbf{u}_{T}} - \sum_{j = 1}^{n}\gamma_{\iota_{j-1},\iota_{j}}(\tau_{j})\mathbf{1}_{\{ \tau_{j} < T \}}\biggm\vert \mathcal{F}_{t}\right],\hspace{1em} t \in [0,T].
	\end{equation}
	Moreover, the sequence $\left\lbrace Y^{i,n}\right\rbrace_{n \ge 0}$ is increasing and converges pointwise $\mathsf{P}$-a.s. for any $0 \le t \le T$ to a c\`{a}dl\`{a}g process $\tilde{Y}^{i}$ satisfying: $\forall t \in [0,T]$,
	\begin{equation}\label{eq:Optimal-Switching-ConvergenceOfSwitchingProcess}
	\tilde{Y}^{i}_{t} = \esssup\limits_{\alpha \in \mathcal{A}_{t,i}}J\left(\alpha;t,i\right) \eqqcolon V(t,i) \quad \text{a.s.}
	\end{equation}
\end{lemma}
\begin{proof}
Let $t \in [0,T]$, $i \in \mathbb{I}$ be given and for $n \ge 0$ define $\mathcal{A}^{n}_{t,i}$ as the subset of admissible strategies with at most $n$ switches:
\[
\mathcal{A}^{n}_{t,i} = \left\lbrace \alpha \in \mathcal{A}_{t,i} \colon \tau_{n+1} = T,\hspace{2bp} \mathsf{P}-a.s. \right\rbrace
\]
Define a double sequence $\hat{\alpha}^{(n)} = \left(\hat{\tau}_{k},\hat{\iota}_{k}\right)^{n+1}_{k = 0}$ as follows
\begin{gather}\label{eq:Optimal-Switching-Optimal-Strategy-n-switches}
\hat{\tau}_{0} = t,\hspace{1em} \hat{\iota}_{0} = i, \nonumber \\
\begin{cases}
\hat{\tau}_{k} = \inf\left\lbrace s \ge \hat{\tau}_{k-1} \colon Y^{\hat{\iota}_{k-1},n-(k-1)}_{s} = \max\limits_{j \neq \hat{\iota}_{k-1}} \left(Y^{j,n-k}_{s} - \gamma_{\hat{\iota}_{k-1},j}\left(s\right) \right)  \right\rbrace \wedge T,\nonumber \\
\hat{\iota}_{k} = \sum\limits_{j \in \mathbb{I}} j \mathbf{1}_{F^{\hat{\iota}_{k-1}}_{j}}
\end{cases} \\
\text{for } k = 1,\ldots,n \text{ where } F^{\hat{\iota}_{k-1}}_{j} \text{ is the event} : \\
F^{\hat{\iota}_{k-1}}_{j} \coloneqq \left\lbrace Y^{j,n-k}_{\hat{\tau}_{k}} -\gamma_{\hat{\iota}_{k-1},j}\left(\hat{\tau}_{k}\right) = \max\limits_{\ell \neq \hat{\iota}_{k-1}}\left( Y^{\ell,n-k}_{\hat{\tau}_{k}} -\gamma_{\hat{\iota}_{k-1},\ell}\left(\hat{\tau}_{k}\right) \right) \right\rbrace, \nonumber
\end{gather}
and set $\hat{\tau}_{n+1} = T,~\hat{\iota}_{n+1}(\omega) = j \in \mathbb{I}$ with $j \neq \hat{\iota}_{n}(\omega)$. Since $Y^{i,n} \in \mathcal{Q} \cap \mathcal{S}^{2}$, one verifies that $\hat{\alpha}^{(n)} \in \mathcal{A}^{n}_{t,i}$ and, using the arguments of Theorem~\ref{Theorem:Optimal-Switching-VerificationPartial}, that $Y^{i,n}_{t} = J(\hat{\alpha}^{(n)};t,i)$ and has the representation \eqref{eq:Optimal-Switching-AtMostnSwitches}. Furthermore, since $\mathcal{A}^{n}_{t,i} \subset \mathcal{A}^{n+1}_{t,i} \subset \mathcal{A}_{t,i}$, it follows that $Y^{i,n}_{t}$ is non-decreasing in $n$ for all $t \in [0,T]$ and $Y^{i,n}_{t} \le Y^{i,n+1}_{t} \le V(t,i)$ almost surely. Recalling also the processes $\hat{U}^{i,n}$ and $\hat{Y}^{i,n}$ from Lemma~\ref{Lemma:Optimal-Switching-FinitelyManySwitchesQLC_S2} and that $\hat{Y}^{i,n}$ is the Snell envelope of $\hat{U}^{i,n}$ for each $n \ge 0$, we deduce $\lbrace \hat{Y}^{i,n} \rbrace_{n \ge 0}$ is an increasing sequence of c\`{a}dl\`{a}g supermartingales. Theorem \MakeUppercase{\romannumeral 6}.18 of \cite{Dellacherie1982} shows that this sequence converges to a limit $\hat{Y}^{i}$ defined pointwise on $[0,T]$ by
\begin{equation*}\label{eq:Continuous-Time-Optimal-Switching-Limiting-Snell-Envelope}
\hat{Y}^{i}_{t} \coloneqq \sup_{n}\hat{Y}^{i,n}_{t} = \sup_{n}\left(Y^{i,n}_{t} + \int_{0}^{t}\psi_{i}(s){d}s\right).
\end{equation*}
This random function $\hat{Y}^{i} = (\hat{Y}^{i}_{t})_{0 \le t \le T}$ is indistinguishable from a c\`{a}dl\`{a}g process, but is not necessarily a supermartingale since we have not established its integrability. Nevertheless, the sequence $\lbrace Y^{i,n} \rbrace_{n \ge 0}$ converges pointwise on $[0,T]$ to a limit $\tilde{Y}^{i}$ which, modulo indistinguishability, is a c\`{a}dl\`{a}g process given by
\begin{equation}\label{eq:Continuous-Time-Optimal-Switching-LimitingProcess}
\tilde{Y}^{i}_{t} = \sup_{n}Y^{i,n}_{t} = \hat{Y}^{i}_{t} - \int_{0}^{t}\psi_{i}(s){d}s.
\end{equation}

Next, let $\alpha = (\tau_{k},\iota_{k})_{k \ge 0} \in \mathcal{A}_{t,i}$ be arbitrary. By Remark~\ref{Remark:SubOptimalSwitchTwice}, we can restrict our attention to those strategies such that $\mathsf{P}\big( \lbrace \tau_{k} = \tau_{k+1}, \tau_{k} < T \rbrace) = 0$ for $k \ge 1$. Define $\alpha^{n} = (\tau^{n}_{k},\iota^{n}_{k})_{k \ge 0}$ to be the strategy obtained from $\alpha$ when only the first $n$ switches are kept:
\[
\begin{cases}
(\tau^{n}_{k},\iota^{n}_{k}) = (\tau_{k},\iota_{k}),\hspace{1em} k \le n, \\
\tau^{n}_{k} = T,\hspace{1em} k > n.
\end{cases}
\]
The difference between the performance indices under $\alpha$ and $\alpha^{n}$ is:
\begin{align*}
J(\alpha;t,i) - J(\alpha^{n};t,i) & = 
\mathsf{E}\left[\int_{\tau_{n}}^{T}\left(\psi_{\mathbf{u}_{s}}(s) - \psi_{\iota^{n}_{n}}(s)\right){d}s + \Gamma_{\mathbf{u}_{T}} - \Gamma_{\iota^{n}_{n}}
- \sum_{k > n}\gamma_{\iota_{k-1},\iota_{k}}(\tau_{k})\mathbf{1}_{\{\tau_{k} < T\}} \biggm\vert \mathcal{F}_{t}\right] \nonumber \\
& = \mathsf{E}\left[\int_{\tau_{n}}^{T}\left(\psi_{\mathbf{u}_{s}}(s) - \psi_{\iota^{n}_{n}}(s)\right){d}s + \Gamma_{\mathbf{u}_{T}} - \Gamma_{\iota^{n}_{n}}
- (C^{\alpha} - C^{\alpha}_{n}) \biggm\vert \mathcal{F}_{t}\right]
\end{align*}
where $\mathbf{u}$ is the mode indicator function associated with $\alpha$ and $\iota^{n}_{n} = \iota_{n \wedge N(\alpha)}$ is the last mode switched to before $T$ under $\alpha^{n}$. Since $\alpha \in \mathcal{A}_{t,i}$, $\psi_{i} \in \mathcal{M}^{2}$ and $\Gamma_{i} \in L^{2}$ for every $i \in \mathbb{I}$, the conditional expectation above is well-defined for every $n \ge 1$. This also leads to an integrable upper bound for $J(\alpha;t,i)$,
\begin{equation}\label{eq:Optimal-Switching-DifferenceBetweenPerformanceIndices2}
\begin{split}
J(\alpha;t,i) \le {} & \mathsf{E}\left[\left(\int_{\tau_{n}}^{T}\left|\psi_{\mathbf{u}_{s}}(s) - \psi_{\iota^{n}_{n}}(s)\right|{d}s + |\Gamma_{\mathbf{u}_{T}} - \Gamma_{\iota^{n}_{n}}| + \left|C^{\alpha} - C^{\alpha}_{n}\right|\right)\mathbf{1}_{\lbrace N(\alpha) > n \rbrace} \biggm\vert \mathcal{F}_{t}\right]\\
& + J(\alpha^{n};t,i)
\end{split}
\end{equation}

Using these integrability conditions again together with the observation that $N(\alpha) < \infty$ $\mathsf{P}$-a.s. and $\lbrace \tau_{k} \rbrace$ is (strictly) increasing towards $T$, we may pass to the limit $n \to \infty$ in equation~\eqref{eq:Optimal-Switching-DifferenceBetweenPerformanceIndices2} to get,
\begin{equation}\label{eq:Optimal-Switching-DifferenceBetweenPerformanceIndices3}
J(\alpha;t,i) \le \lim\limits_{n \to \infty}J(\alpha^{n};t,i) \quad \text{a.s.}
\end{equation}
However, as $\alpha^{n} \in \mathcal{A}^{n}_{t,i}$ for each $n \ge 0$, from \eqref{eq:Optimal-Switching-DifferenceBetweenPerformanceIndices3} and \eqref{eq:Continuous-Time-Optimal-Switching-LimitingProcess} we get for every $t \in [0,T]$:
\[
J(\alpha;t,i) \le \lim\limits_{n \to \infty}J(\alpha^{n};t,i) \le \lim\limits_{n \to \infty}Y^{i,n}_{t} = \tilde{Y}^{i}_{t} \quad \text{a.s.}
\]
Since $\alpha \in \mathcal{A}_{t,i}$ was arbitrary, we have just shown for every $t \in [0,T]$
\[
V(t,i) \coloneqq \esssup\limits_{\alpha \in \mathcal{A}_{t,i}}J\left(\alpha;t,i\right) \le \tilde{Y}^{i}_{t}\quad \text{a.s.}
\]
The reverse inequality holds since $Y^{i,n}_{t} = J(\hat{\alpha}^{(n)};t,i) \le V(t,i)$ almost surely for $n \ge 0$ (cf. \eqref{eq:Optimal-Switching-Optimal-Strategy-n-switches}) and $\tilde{Y}^{i}$ is the pointwise supremum of the sequence $\{Y^{i,n}\}_{n \ge 0}$.
 \end{proof}
\subsection{The case of an arbitrary number of switches.}\label{Section:ArbitraryNumberOfSwitches}
This section gives sufficient conditions under which the limiting processes $\tilde{Y}^{1},\ldots,\tilde{Y}^{m}$ satisfy the verification theorem~\ref{Theorem:Optimal-Switching-VerificationPartial}. The main difficulty is in proving that $\tilde{Y}^{i} \in \mathcal{S}^{2}$, and in order to achieve this we make the following hypothesis.
\begin{quote}
	\vskip1em
	\textbf{(M)} There exists a family of martingales $\{M_{ij} = (M_{ij})_{0 \le t \le T} \colon i,j \in \mathbb{I}\}$ such that for every $i,j,k \in \mathbb{I}$:
	\begin{align*}
	i. \quad & M_{i,j} \in \mathcal{S}^{2}  \nonumber \\
	ii. \quad & -\gamma_{i,j}(\cdot) \le M_{i,j}(\cdot), \enskip \mathsf{P}-\text{a.s.} \enskip \text{ if } i \neq j \nonumber \\
	iii. \quad & M_{i,j}(\cdot) + M_{j,k}(\cdot) \le M_{i,k}(\cdot), \enskip \mathsf{P}-\text{a.s.} \enskip \text{ if } i \neq j \text{ and } j \neq k.
	\end{align*}
\end{quote}
This hypothesis can be verified in the following cases:
\begin{itemize}
	\item The switching costs are martingales -- since we can set $M_{i,j} = -\gamma_{i,j}$ (with strict inequality in property iii. above). This includes the case $\gamma_{i,j}(t) = \gamma_{i,j}$, $t \in [0,T]$, with $\gamma_{i,j} \in L^{2}$ and $\mathcal{F}_{0}$-measurable;
	\item The switching costs are non-negative -- since we can set $M_{i,j} \equiv 0$ for $i,j \in \mathbb{I}$;
	\item There are two modes ($\mathbb{I} = \{0,1\}$ as per convention). For $i \in \{0,1\}$ and $j = 1-i$, let $Z_{i,j}$ denote the Snell envelope of $(-\gamma_{i,j}(t))_{0 \le t \le T}$, which exists and is in $\mathcal{S}^{2}$ since $\gamma_{i,j} \in \mathcal{S}^{2}$ (see Proposition~\ref{Proposition:Optimal-Switching-SnellEnvelopeProperties}). We may then take $M_{i,j}$ to be the martingale component in the Doob-Meyer decomposition of $Z_{i,j}$, and set $M_{0,0} = M_{1,1} = M_{0,1} + M_{1,0}$. This case includes many examples of \emph{Dynkin games} (see \cite{Martyr2014c}).
\end{itemize}

\begin{lemma}\label{Lemma:Martingale-Hypothesis-Lemma}
	Assume Hypothesis \textbf{(M)}, then $\forall \alpha = (\tau_{n},\iota_{n})_{n \ge 0} \in \mathcal{A}_{t,i}$, $(t,i) \in [0,T] \times \mathbb{I}$:
	\begin{equation}\label{eq:Optimal-Switching-Constant-Upper-Limit}
	\forall N \ge 1, \quad \mathsf{E}\left[-\sum_{n = 1}^{N}\gamma_{\iota_{n-1},\iota_{n}}(\tau_{n}) \biggm\vert \mathcal{F}_{t}\right] \le \mathsf{E}\left[\max_{j_{1},j_{2} \in \mathbb{I}}|M_{j_{1},j_{2}}(T)| \biggm \vert \mathcal{F}_{t}\right] \hspace{1em}\mathsf{P}-\text{a.s.}
	\end{equation}
\end{lemma}
\begin{proof}
Let $\alpha = (\tau_{n},\iota_{n})_{n \ge 0} \in \mathcal{A}_{t,i}$ be arbitrary. For $n \ge 1$ and $i,j \in \mathbb{I}$ we have $\tau_{n} \le T$, $\mathbf{1}_{\{\iota_{n-1} = i\}}\mathbf{1}_{\{\iota_{n} = j\}}$ is non-negative and $\mathcal{F}_{\tau_{n}}$-measurable, $M_{ij} \in \mathcal{S}^{2}$ is a martingale with $-\gamma_{i,j}(\cdot) \le M_{i,j}(\cdot)$ for $i \neq j$. We can therefore show for $N \ge 1$:
\[
\mathsf{E}\left[-\sum_{n = 1}^{N}\gamma_{\iota_{n-1},\iota_{n}}(\tau_{n}) \biggm\vert \mathcal{F}_{t}\right] \le \mathsf{E}\left[\sum_{n = 1}^{N}M_{\iota_{n-1},\iota_{n}}(\tau_{n}) \biggm\vert \mathcal{F}_{t}\right] = \mathsf{E}\left[\sum_{n = 1}^{N} M_{\iota_{n-1},\iota_{n}}(T) \biggm\vert \mathcal{F}_{t}\right]
\]
The proof can be completed by showing
\begin{equation}\label{eq:Optimal-Switching-Constant-Upper-Limit-Proof}
\forall N \ge 1, \quad \sum_{n = 1}^{N} M_{\iota_{n-1},\iota_{n}}(T) \le \max_{j_{1},j_{2} \in \mathbb{I}}|M_{j_{1},j_{2}}(T)| \hspace{1em}\mathsf{P}-\text{a.s.}
\end{equation}
and concluding by arbitrariness of $\alpha$. The inequality \eqref{eq:Optimal-Switching-Constant-Upper-Limit-Proof} shall be proved via induction similarly to \cite[p.~399]{LyVath2007}. First note that \eqref{eq:Optimal-Switching-Constant-Upper-Limit-Proof} is true for $N = 1$. Now, suppose that \eqref{eq:Optimal-Switching-Constant-Upper-Limit-Proof} is satisfied for $N \ge 1$. Since $M_{\iota_{N-1},\iota_{N}}(T) + M_{\iota_{N},\iota_{N+1}}(T) \le M_{\iota_{N-1},\iota_{N+1}}(T)$ a.s. we have
\[
\sum_{n = 1}^{N+1}M_{\iota_{n-1},\iota_{n}}(T) \le \sum_{n = 1}^{N-1}M_{\iota_{n-1},\iota_{n}}(T) + M_{\iota_{N-1},\iota_{N+1}}(T) \enskip\mathsf{P}-\text{a.s.}
\]
Define a new strategy $\tilde{\alpha} = (\tilde{\tau}_{n},\tilde{\iota}_{n})_{n \ge 0} \in \mathcal{A}_{t,i}$ by $(\tilde{\tau}_{n},\tilde{\iota}_{n}) = (\tau_{n},\iota_{n})$ for $n = 1,\ldots,N-1$ and $(\tilde{\tau}_{n},\tilde{\iota}_{n}) = (\tau_{n+1},\iota_{n+1})$ for $n \ge N$. Then, using the induction hypothesis on $\tilde{\alpha}$, one gets
\[
\sum_{n = 1}^{N+1}M_{\iota_{n-1},\iota_{n}}(T) \le \sum_{n = 1}^{N}M_{\tilde{\iota}_{n-1},\tilde{\iota}_{n}}(T) \le \max_{j_{1},j_{2} \in \mathbb{I}}|M_{j_{1},j_{2}}(T)| \enskip\mathsf{P}-\text{a.s.}
\]
 \end{proof}

\begin{theorem}[Existence]\label{Theorem:Optimal-Switching-ExistenceVerification}
	Suppose Hypothesis \textbf{(M)}. Then the limit processes $\tilde{Y}^{1},\ldots,\tilde{Y}^{m}$ of Lemma~\ref{Lemma:OptimalSwitching-ConvergenceOfTheOptimalProcesses} satisfy the following: for $i \in \mathbb{I}$,
	\begin{enumerate}
		\item $\tilde{Y}^{i} \in \mathcal{Q} \cap \mathcal{S}^{2}$.
		\item For any $0 \le t \le T$,
		\begin{equation}\label{eq:Optimal-Switching-LimitingSnellEnvelope}
		\begin{split}
		\tilde{Y}^{i}_{t} & = \esssup\limits_{\tau \ge t}\mathsf{E}\left[\int_{t}^{\tau}\psi_{i}(s){d}s + \Gamma_{i}\mathbf{1}_{\{\tau = T \}} + \max_{j \neq i}\left\lbrace \tilde{Y}^{j}_{\tau} - \gamma_{i,j}(\tau) \right\rbrace\mathbf{1}_{\{\tau < T \}}\biggm\vert \mathcal{F}_{t}\right],\\
		\tilde{Y}^{i}_{T} & = \Gamma_{i}.
		\end{split}
		\end{equation}
	\end{enumerate}
	In particular, $\tilde{Y}^{1},\ldots,\tilde{Y}^{m}$ are unique and satisfy the verification theorem.
\end{theorem}
\begin{proof}
Recall the limit processes $\hat{Y}^{1},\ldots,\hat{Y}^{m}$ and $\tilde{Y}^{1},\ldots,\tilde{Y}^{m}$ from Lemma~\ref{Lemma:Optimal-Switching-FinitelyManySwitchesQLC_S2}, equation~\eqref{eq:Continuous-Time-Optimal-Switching-LimitingProcess}. Under Hypothesis \textbf{(M)} one verifies directly using Lemma~\ref{Lemma:Martingale-Hypothesis-Lemma} and the arguments in Lemma~\ref{Lemma:OptimalSwitching-ConvergenceOfTheOptimalProcesses} that the $\mathbb{F}$-martingale $\zeta = (\zeta_{t})_{0 \le t \le T}$ defined by
\begin{equation}\label{eq:Optimal-Switching-BoundOnMartingale}
\zeta_{t} \coloneqq \mathsf{E}\left[\int_{0}^{T}\max_{j \in \mathbb{I}}\left|\psi_{j}(s)\right|{d}s + \max_{j \in \mathbb{I}}|\Gamma_{j}| + \max_{j_{1},j_{2} \in \mathbb{I}}|M_{j_{1},j_{2}}(T)| \biggm \vert \mathcal{F}_{t}\right]
\end{equation}
satisfies $\zeta \in \mathcal{S}^{2}$ and $\forall n \ge 0$, $|\hat{Y}^{i,n}_{t}| \le \zeta_{t}$ $\mathsf{P}$-a.s. for every $t \in [0,T]$. Moreover, since $\hat{Y}^{i}$ is the pointwise supremum of $\lbrace \hat{Y}^{i,n} \rbrace_{n \ge 0}$ we also have $\hat{Y}^{i}_{t} \le \zeta_{t}$ for each $t \in [0,T]$. These observations give $-\zeta_{t} \le \hat{Y}^{i}_{t} \le \zeta_{t}$ $\mathsf{P}-\text{a.s. } \forall~ t \in [0,T]$. Since $\zeta \in \mathcal{S}^{2}$, it follows that $\hat{Y}^{i} \in \mathcal{S}^{2}$ and also $\tilde{Y}^{i} \in \mathcal{S}^{2}$ since $\psi_{i} \in \mathcal{M}^{2}$.

Now define a process $\hat{U}^{i} = (\hat{U}^{i}_{t})_{0 \le t \le T}$ for $i = 1,\ldots,m$ similarly to $\hat{U}^{i,n}$ used in Lemma~\ref{Lemma:Optimal-Switching-FinitelyManySwitchesQLC_S2}:
\begin{align*}
\hat{U}^{i}_{t} \coloneqq \int_{0}^{t}\psi_{i}(s){d}s + \Gamma_{i}\mathbf{1}_{\{ t = T \}} + \max_{j \neq i}\left\lbrace \tilde{Y}^{j}_{t} - \gamma_{i,j}(t) \right\rbrace\mathbf{1}_{\{ t < T \}}
\end{align*}
The $\mathcal{S}^{2}$ processes $\hat{Y}^{i}$ and $\hat{U}^{i}$ are the respective limits of the increasing sequences of c\`{a}dl\`{a}g $\mathcal{S}^{2}$ processes $\lbrace \hat{Y}^{i,n} \rbrace_{n \ge 0}$ and $\lbrace \hat{U}^{i,n} \rbrace_{n \ge 0}$. Since $\hat{Y}^{i,n}$ is also the Snell envelope of $\hat{U}^{i,n}$, property 5 of Proposition~\ref{Proposition:Optimal-Switching-SnellEnvelopeProperties} verifies that $\hat{Y}^{i}$ is the Snell envelope of $\hat{U}^{i}$. This leads to equation~\eqref{eq:Optimal-Switching-LimitingSnellEnvelope} for $\tilde{Y}^{i}$ and the uniqueness claim.

The final part is to show that $\tilde{Y}^{i} \in \mathcal{Q}$. Let $\tau \in \mathcal{T}$ be any predictable time. Since $\hat{Y}^{i}$ is the Snell envelope of $\hat{U}^{i}$, it has a Meyer decomposition (cf. Proposition~\ref{Proposition:Optimal-Switching-SnellEnvelopeProperties})
\[
\hat{Y}^{i} = M - A - B,
\]
where $M$ is a uniformly integrable c\`{a}dl\`{a}g martingale and $A$ (resp. $B$) is non-decreasing, predictable and continuous (resp. discontinuous). Remember that $M$ is also quasi-left-continuous due to Assumption~\ref{Assumption:QLCFiltration} and Proposition~\ref{Proposition:Optimal-Switching-qlcFiltrations}. We therefore have
\[
\triangle_{\tau}\hat{Y}^{i} = \left(M_{\tau} - A_{\tau} - B_{\tau}\right) - \left(M_{\tau^{-}} - A_{\tau^{-}} - B_{\tau^{-}}\right) = -\triangle_{\tau}B \quad \text{a.s.}
\]
By property 2 of Proposition~\ref{Proposition:Optimal-Switching-SnellEnvelopeProperties} concerning the jumps of $\hat{Y}^{i}$ (and therefore $\tilde{Y}^{i}$), we have
\[
\lbrace \triangle_{\tau}B > 0 \rbrace \subset \lbrace \hat{Y}^{i}_{\tau^{-}} = \hat{U}^{i}_{\tau^{-}} \rbrace
\]
and by using the definitions of $\hat{Y}^{i}$ and $\hat{U}^{i}$ we get:
\begin{equation}\label{eq:Optimal-Switching-JumpInIndexi}
\tilde{Y}^{i}_{\tau} < \tilde{Y}^{i}_{\tau^{-}} = \max_{j \neq i}\left\lbrace \tilde{Y}^{j}_{\tau^{-}} - \gamma_{i,j}(\tau^{-}) \right\rbrace \text{ on } \lbrace \triangle_{\tau}B > 0 \rbrace.
\end{equation}
Since $\mathbb{I}$ is finite, \eqref{eq:Optimal-Switching-JumpInIndexi} implies that there exists an $\mathbb{I}$-valued random variable $j^{*}$, $j^{*} \neq i$, such that
\begin{equation}\label{eq:Optimal-Switching-ContradictionAtJump1}
\tilde{Y}^{i}_{\tau^{-}} = \tilde{Y}^{j^{*}}_{\tau^{-}} - \gamma_{i,j^{*}}(\tau^{-}) = \max_{j \neq i}\left\lbrace \tilde{Y}^{j}_{\tau^{-}} -\gamma_{i,j}(\tau^{-}) \right\rbrace \text{ on } \lbrace \triangle_{\tau}B > 0 \rbrace.
\end{equation}
However, \eqref{eq:Optimal-Switching-JumpInIndexi} also implies that the process $\left(\max_{j \neq i}\left\lbrace \tilde{Y}^{j}_{t} - \gamma_{i,j}(t)\right\rbrace\right)_{0 \le t \le T}$ jumps at time $\tau$ (since it is dominated by $\tilde{Y}^{i}$). As the switching costs are quasi-left-continuous, we conclude that $\tilde{Y}^{j^{*}}$ jumps at time $\tau$. Using the Meyer decomposition of $\hat{Y}^{j^{*}}$ and the properties of the jumps as before, this leads to
\[
\tilde{Y}^{j^{*}}_{\tau^{-}} = \max_{l \neq j^{*}}\left\lbrace \tilde{Y}^{l}_{\tau^{-}} - \gamma_{j^{*},l}(\tau^{-}) \right\rbrace \text{ on } \lbrace \triangle_{\tau}B > 0 \rbrace
\]
and there exists an $\mathbb{I}$-valued random variable $l^{*}$, $l^{*} \neq j^{*}$, such that
\begin{equation}\label{eq:Optimal-Switching-ContradictionAtJump2}
\tilde{Y}^{j^{*}}_{\tau^{-}} = \tilde{Y}^{l^{*}}_{\tau^{-}} - \gamma_{j^{*},l^{*}}(\tau^{-}) = \max_{l \neq j^{*}}\left\lbrace \tilde{Y}^{l}_{\tau^{-}} - \gamma_{j^{*},l}(\tau^{-}) \right\rbrace \text{ on } \lbrace \triangle_{\tau}B > 0 \rbrace
\end{equation}

Putting \eqref{eq:Optimal-Switching-ContradictionAtJump1} and \eqref{eq:Optimal-Switching-ContradictionAtJump2} together, then using the quasi-left-continuity of the switching costs and Assumption~\ref{Assumption:SwitchingCosts}, the following (almost sure) inequality and contradiction to the optimality of $j^{*}$ is obtained:
\begin{align*}
\tilde{Y}^{i}_{\tau^{-}} = -\gamma_{i,j^{*}}(\tau^{-}) + \tilde{Y}^{j^{*}}_{\tau^{-}} & = -\gamma_{i,j^{*}}(\tau^{-}) -\gamma_{j^{*},l^{*}}(\tau^{-}) + \tilde{Y}^{l^{*}}_{\tau^{-}} \\
& = -\gamma_{i,j^{*}}(\tau) -\gamma_{j^{*},l^{*}}(\tau) + \tilde{Y}^{l^{*}}_{\tau^{-}} \\
& < -\gamma_{i,l^{*}}(\tau) + \tilde{Y}^{l^{*}}_{\tau^{-}} \\
& < -\gamma_{i,l^{*}}(\tau^{-}) + \tilde{Y}^{l^{*}}_{\tau^{-}} \enskip \text{on} \enskip \lbrace \triangle_{\tau}B > 0 \rbrace.
\end{align*}
This means $\triangle_{\tau}B = 0$ a.s. for every predictable time $\tau$, and $\tilde{Y}^{i} \in \mathcal{Q}$ for every $i \in \mathbb{I}$.\end{proof}

\section{Conclusion.}\label{Section:Optimal-Switching-Continuous-Time-Conclusion}
This paper extended the study of the multiple modes optimal switching problem in \cite{Djehiche2009} to account for
\begin{enumerate}
	\item non-zero, possibly different terminal rewards;
	\item signed switching costs modelled by c\`{a}dl\`{a}g, quasi-left-continuous processes;
	\item filtrations which are only assumed to satisfy the usual conditions and quasi-left-continuity.
\end{enumerate}

Just as in Theorem 1 of \cite{Djehiche2009}, it was shown that the value function of the optimal switching problem can be defined stochastically in terms of interconnected Snell envelope-like processes. The existence of these processes was proved in a manner similar to Theorem 2 of \cite{Djehiche2009}, by a limiting argument for sequences of processes solving the optimal switching problem with at most $n \ge 0$ switches. The limits of these sequences are right-continuous processes, but may not satisfy the integrability assumptions of the Snell envelope representation in general. Sufficient conditions for this representation were obtained by further hypothesizing the existence of a family of martingales satisfying particular relations among themselves and the switching costs. We explained that this ``martingale hypothesis'' can be verified quite easily in the following cases:
\begin{itemize}
	\item the switching costs are martingales;
	\item the switching costs are non-negative;
	\item the case of two modes (starting and stopping problem).
\end{itemize}

\appendix
\numberwithin{equation}{section}
\section{Admissibility of the candidate optimal strategy.}\label{Section:Optimal-Switching-IntegrabilityOfTheSwitchingCost}
Let $\alpha^{*} = \left(\tau^{*}_{n},\iota^{*}_{n}\right)_{n \ge 0}$ be the sequence of times and random mode indicators defined in equation~\eqref{eq:Optimal-Switching-OptimalStoppingStrategy} of Theorem~\ref{Theorem:Optimal-Switching-VerificationPartial}. In this section we prove that $\alpha^{*} \in \mathcal{A}_{t,i}$ (cf. Definition~\ref{Definition:Optimal-Switching-AdmissibleStrategies}). One readily verifies (by right-continuity) that $\left\lbrace\tau^{*}_{n}\right\rbrace_{n \ge 0} \subset \mathcal{T}$ is non-decreasing with $\tau^{*}_{0} = t$, and each $\iota^{*}_{n}$ is an $\mathcal{F}_{\tau^{*}_{n}}$-measurable $\mathbb{I}$-valued random variable with $\iota^{*}_{0} = i$ and $\iota^{*}_{n} \neq \iota^{*}_{n+1}$ for $n \ge 0$. The remaining properties are established in a number of steps, beginning with the following lemma on the switching times.

\begin{lemma}\label{Lemma:Optimal-Switching-FiniteStrategyCandidateOptimal}
	Let $\left\lbrace\tau^{*}_{n}\right\rbrace_{n \ge 0}$ be the switching times defined in equation~\eqref{eq:Optimal-Switching-OptimalStoppingStrategy} of Theorem~\ref{Theorem:Optimal-Switching-VerificationPartial}. Then these times satisfy
	\begin{equation}\label{eq:FiniteAndStrictlyIncreasingSwitchingTimes}
	\begin{split}
	i) & \quad \mathsf{P}(\{\tau^{*}_{n} = \tau^{*}_{n+1}, \tau^{*}_{n} < T\}) = 0,\quad n \ge 1  \\
	ii) & \quad \mathsf{P}\left(\{\tau^{*}_{n} < T,\hspace{1bp} \forall n \ge 0\}\right) = 0
	\end{split}
	\end{equation}
\end{lemma}
\begin{proof}
Condition~\eqref{eq:FiniteAndStrictlyIncreasingSwitchingTimes}-i) can be proved via contradiction using Assumption~\ref{Assumption:SwitchingCosts} (recall Remark~\ref{Remark:SubOptimalSwitchTwice}). Condition~\eqref{eq:FiniteAndStrictlyIncreasingSwitchingTimes}-ii) can also be proved by contradiction using Assumption~\ref{Assumption:SwitchingCosts} and the same arguments of \cite[pp.~192--193]{Hamadene2012} (since the switching costs are quasi-left-continuous). The details are therefore omitted.
 \end{proof}

The rest of this section is devoted to verifying condition~\eqref{eq:Optimal-Switching-UISupremumClosureSwitchingCosts} for the strategy $\alpha^{*}$. Recall that the cumulative cost of switching $n \ge 1$ times is given by,
\[
C^{\alpha^{*}}_{n} = \sum\limits_{k=1}^{n}\gamma_{\iota^{*}_{k-1},\iota^{*}_{k}}(\tau^{*}_{k})\mathbf{1}_{\{\tau^{*}_{k} < T\}}
\]
Since the switching costs satisfy $\gamma_{i,j} \in \mathcal{S}^{2}$ for every $i,j$ in the finite set $\mathbb{I}$, $C^{\alpha^{*}}_{n} \in L^{2}$ for every $n \ge 1$. We define a sequence 
\[
N^{*}_{n} \coloneqq \sum_{k=1}^{n}\mathbf{1}_{\{\tau^{*}_{k} < T\}},\quad n = 1,2,\ldots
\]
which we use to rewrite the expression for $C^{\alpha^{*}}_{n}$ as follows:
\begin{equation}\label{eq:Optimal-Switching-CostOfSwitchingFinitelyManyTimes2}
C^{\alpha^{*}}_{n} = \sum\limits_{k=1}^{N^{*}_{n}}\gamma_{\iota^{*}_{k-1},\iota^{*}_{k}}(\tau^{*}_{k}).
\end{equation}
The following proposition gives an alternative representation of $C^{\alpha^{*}}_{n}$ in terms of the processes $Y^{1},\ldots,Y^{m}$ and their Meyer decomposition with random superscripts (cf. Lemma~\ref{Lemma:Optimal-Switching-VerificationLemma}).

\begin{proposition}\label{Proposition:Optimal-Switching-AdmissibilityCandidateOptimalStrategy}
	Let $\alpha^{*} = \left(\tau^{*}_{n},\iota^{*}_{n}\right)_{n \ge 0} \in \mathcal{A}_{t,i}$ be the switching control strategy defined in equation~\eqref{eq:Optimal-Switching-OptimalStoppingStrategy} of Theorem~\ref{Theorem:Optimal-Switching-VerificationPartial} and let $\mathbf{u}^{*}$ be the associated mode indicator function. Then $C^{\alpha^{*}}_{n}$, the cumulative cost of switching $n \ge 1$ times under $\alpha^{*}$, satisfies
	\begin{equation}\label{eq:Optimal-Switching-CumulativeCostOfSwitchingMeyerDecompositionFinal}
	C^{\alpha^{*}}_{n} = Y^{\iota^{*}_{N^{*}_{n}}}_{\tau^{*}_{N^{*}_{n}}} - Y^{\iota^{*}_{0}}_{\tau^{*}_{0}} +  \int_{\tau^{*}_{0}}^{\tau^{*}_{N^{*}_{n}}}\psi_{\mathbf{u}^{*}_{s}}(s){d}s - \sum_{k = 1}^{N^{*}_{n}}\left(M^{\iota^{*}_{k-1}}_{\tau^{*}_{k}} - M^{\iota^{*}_{k-1}}_{\tau^{*}_{k-1}}\right) \hspace{1em} \mathsf{P}-\text{a.s.}
	\end{equation}
	where $M^{\iota^{*}_{k}}$, $k \ge 0$, is the martingale component of the Meyer decomposition \eqref{eq:Optimal-Switching-VerificationLemmaMeyerDecomposition} in Lemma~\ref{Lemma:Optimal-Switching-VerificationLemma}.
\end{proposition}
\begin{proof}
By definition of the strategy $\alpha^{*}$ (cf. \eqref{eq:Optimal-Switching-OptimalStoppingStrategy}), optimality of the time $\tau^{*}_{n}$ and the definition of $\iota^{*}_{n}$, for $n \ge 1$ the cost of switching at $\tau^{*}_{n}$ is,
\begin{equation}\label{eq:Optimal-Switching-CostOfSwitching}
\gamma_{\iota^{*}_{n-1},\iota^{*}_{n}}(\tau^{*}_{n})\mathbf{1}_{\{\tau^{*}_{n} < T\}} = \left(Y^{\iota^{*}_{n}}_{\tau^{*}_{n}} - Y^{\iota^{*}_{n-1}}_{\tau^{*}_{n}}\right)\mathbf{1}_{\{\tau^{*}_{n} < T\}}\hspace{1em}\mathsf{P}-\text{a.s.}
\end{equation}
Therefore, from equation~\eqref{eq:Optimal-Switching-CostOfSwitchingFinitelyManyTimes2} and \eqref{eq:Optimal-Switching-CostOfSwitching} the cost of the first $n$ switches can be rewritten as,
\begin{equation}\label{eq:Optimal-Switching-CumulativeCostOfSwitching}
C^{\alpha^{*}}_{n} = \sum_{k = 1}^{N^{*}_{n}}\left(Y^{\iota^{*}_{k}}_{\tau^{*}_{k}} - Y^{\iota^{*}_{k-1}}_{\tau^{*}_{k}}\right)\hspace{1em}\mathsf{P}-\text{a.s.}
\end{equation}
Now, Lemma~\ref{Lemma:Optimal-Switching-VerificationLemma} proved that the following Meyer decomposition holds for $k \ge 0$ (cf. equation~\eqref{eq:Optimal-Switching-VerificationLemmaMeyerDecomposition}):
\begin{equation}\label{eq:Optimal-Switching-OptimalProcessMeyerDecomposition}
Y^{\iota^{*}_{k}}_{t} + \int_{0}^{t}\psi_{\iota^{*}_{k}}(s){d}s = M^{\iota^{*}_{k}}_{t} - A^{\iota^{*}_{k}}_{t},\hspace{1em}\mathsf{P}-\text{a.s. }\forall \hspace{2bp} \tau^{*}_{k} \le t \le T.
\end{equation}
where, on $[\tau^{*}_{k},T]$, $M^{\iota^{*}_{k}}$ is a uniformly integrable c\`{a}dl\`{a}g martingale and $A^{\iota^{*}_{k}}$ is a predictable, continuous and increasing process. The Meyer decomposition is used to rewrite equation~\eqref{eq:Optimal-Switching-CumulativeCostOfSwitching} for the cumulative switching costs as follows: $\mathsf{P}$-a.s.,
\begin{equation}\label{eq:Optimal-Switching-CumulativeCostOfSwitchingMeyerDecomposition}
\begin{split}
C^{\alpha^{*}}_{n} = {} & \sum_{k = 1}^{N^{*}_{n}}\left(M^{\iota^{*}_{k}}_{\tau^{*}_{k}} - M^{\iota^{*}_{k-1}}_{\tau^{*}_{k}}\right) - \sum_{k = 1}^{N^{*}_{n}}\left(A^{\iota^{*}_{k}}_{\tau^{*}_{k}} - A^{\iota^{*}_{k-1}}_{\tau^{*}_{k}}\right) \\
& \qquad - \sum_{k = 1}^{N^{*}_{n}}\left(\int_{0}^{\tau^{*}_{k}}\psi_{\iota^{*}_{k}}(s){d}s - \int_{0}^{\tau^{*}_{k}}\psi_{\iota^{*}_{k-1}}(s){d}s \right).
\end{split}
\end{equation}
The first summation term in equation~\eqref{eq:Optimal-Switching-CumulativeCostOfSwitchingMeyerDecomposition} can be rewritten as:
\begin{equation}\label{eq:Optimal-Switching-SimplificationOfASummationTerm1}
\sum_{k = 1}^{N^{*}_{n}}\left(M^{\iota^{*}_{k}}_{\tau^{*}_{k}} - M^{\iota^{*}_{k-1}}_{\tau^{*}_{k}}\right) = M^{\iota^{*}_{N^{*}_{n}}}_{\tau^{*}_{N^{*}_{n}}} - M^{\iota^{*}_{0}}_{\tau^{*}_{0}} - \sum_{k = 1}^{N^{*}_{n}}\left(M^{\iota^{*}_{k-1}}_{\tau^{*}_{k}} - M^{\iota^{*}_{k-1}}_{\tau^{*}_{k-1}}\right)
\end{equation}
For every $k \ge 0$, by the definition of $\tau^{*}_{k+1}$ and property 4 of Proposition~\ref{Proposition:Optimal-Switching-SnellEnvelopeProperties}, we know that $\left(Y^{\iota^{*}_{k}}_{t} + \int_{0}^{t}\psi_{\iota^{*}_{k}}(s){d}s\right)$ is a martingale $\mathsf{P}$-a.s. for every $\tau^{*}_{k} \le t \le \tau^{*}_{k+1}$. By using the Meyer decomposition \eqref{eq:Optimal-Switching-OptimalProcessMeyerDecomposition}, we therefore observe that
$\forall k \ge 0$, $A^{\iota^{*}_{k}}_{t}$ \emph{is constant} $\mathsf{P}$-a.s. $\forall \hspace{2bp} \tau^{*}_{k} \le t \le \tau^{*}_{k+1}$. The summation term in \eqref{eq:Optimal-Switching-CumulativeCostOfSwitchingMeyerDecomposition} with respect to $A^{\iota^{*}_{k-1}}$ can then be simplified as follows,
\begin{equation}\label{eq:Optimal-Switching-SimplificationOfASummationTerm2}
\sum_{k = 1}^{N^{*}_{n}}\left(A^{\iota^{*}_{k}}_{\tau^{*}_{k}} - A^{\iota^{*}_{k-1}}_{\tau^{*}_{k}}\right) = \sum_{k = 1}^{N^{*}_{n}}\left(A^{\iota^{*}_{k}}_{\tau^{*}_{k}} - A^{\iota^{*}_{k-1}}_{\tau^{*}_{k-1}}\right) = A^{\iota^{*}_{N^{*}_{n}}}_{\tau^{*}_{N^{*}_{n}}} - A^{\iota^{*}_{0}}_{\tau^{*}_{0}} \hspace{1em} \mathsf{P}-\text{a.s.}
\end{equation}

By writing out the terms and using the definition of the mode indicator function $\mathbf{u}^{*}$, the third summation term in \eqref{eq:Optimal-Switching-CumulativeCostOfSwitchingMeyerDecomposition} is simplified as follows: $\mathsf{P}$-a.s.,
\begin{align}\label{eq:Optimal-Switching-SimplificationOfASummationTerm3}
{} & -\sum_{k = 1}^{N^{*}_{n}}\left(\int_{0}^{\tau^{*}_{k}}\psi_{\iota^{*}_{k}}(s){d}s - \int_{0}^{\tau^{*}_{k}}\psi_{\iota^{*}_{k-1}}(s){d}s \right) \nonumber \\
= \quad & \int_{0}^{\tau^{*}_{1}}\psi_{\iota^{*}_{0}}(s){d}s + \sum_{k = 1}^{N^{*}_{n}-1}\int_{\tau^{*}_{k}}^{\tau^{*}_{k+1}}\psi_{\iota^{*}_{k}}(s){d}s - \int_{0}^{\tau^{*}_{N^{*}_{n}}}\psi_{\iota^{*}_{N^{*}_{n}}}(s){d}s \nonumber \\
= \quad & \int_{0}^{\tau^{*}_{1}}\psi_{\iota^{*}_{0}}(s){d}s + \int_{\tau^{*}_{1}}^{\tau^{*}_{N^{*}_{n}}}\psi_{\mathbf{u}^{*}_{s}}(s){d}s - \int_{0}^{\tau^{*}_{N^{*}_{n}}}\psi_{\iota^{*}_{N^{*}_{n}}}(s){d}s
\end{align}

Substitute equations~\eqref{eq:Optimal-Switching-SimplificationOfASummationTerm1}, \eqref{eq:Optimal-Switching-SimplificationOfASummationTerm2}, and \eqref{eq:Optimal-Switching-SimplificationOfASummationTerm3} into equation~\eqref{eq:Optimal-Switching-CumulativeCostOfSwitchingMeyerDecomposition} for the cumulative switching cost, then use the Meyer decomposition~\eqref{eq:Optimal-Switching-OptimalProcessMeyerDecomposition} and the definition of $\mathbf{u}^{*}$ to get,
\begin{align*}
C^{\alpha^{*}}_{n} = {} & M^{\iota^{*}_{N^{*}_{n}}}_{\tau^{*}_{N^{*}_{n}}} - M^{\iota^{*}_{0}}_{\tau^{*}_{0}} - \sum_{k = 1}^{N^{*}_{n}}\left(M^{\iota^{*}_{k-1}}_{\tau^{*}_{k}} - M^{\iota^{*}_{k-1}}_{\tau^{*}_{k-1}}\right) - A^{\iota^{*}_{N^{*}_{n}}}_{\tau^{*}_{N^{*}_{n}}} +  A^{\iota^{*}_{0}}_{\tau^{*}_{0}} \nonumber \\
& + \int_{0}^{\tau^{*}_{1}}\psi_{\iota^{*}_{0}}(s){d}s + \int_{\tau^{*}_{1}}^{\tau^{*}_{N^{*}_{n}}}\psi_{\mathbf{u}^{*}_{s}}(s){d}s - \int_{0}^{\tau^{*}_{N^{*}_{n}}}\psi_{\iota^{*}_{N^{*}_{n}}}(s){d}s \nonumber \\
= {} & Y^{\iota^{*}_{N^{*}_{n}}}_{\tau^{*}_{N^{*}_{n}}} - \left(Y^{\iota^{*}_{0}}_{\tau^{*}_{0}} + \int_{0}^{\tau^{*}_{0}}\psi_{\iota^{*}_{0}}(s){d}s\right) + \int_{0}^{\tau^{*}_{1}}\psi_{\iota^{*}_{0}}(s){d}s +  \int_{\tau^{*}_{1}}^{\tau^{*}_{N^{*}_{n}}}\psi_{\mathbf{u}^{*}_{s}}(s){d}s \nonumber \\
& - \sum_{k = 1}^{N^{*}_{n}}\left(M^{\iota^{*}_{k-1}}_{\tau^{*}_{k}} - M^{\iota^{*}_{k-1}}_{\tau^{*}_{k-1}}\right) \nonumber \\
= {} & Y^{\iota^{*}_{N^{*}_{n}}}_{\tau^{*}_{N^{*}_{n}}} - Y^{\iota^{*}_{0}}_{\tau^{*}_{0}} +  \int_{\tau^{*}_{0}}^{\tau^{*}_{N^{*}_{n}}}\psi_{\mathbf{u}^{*}_{s}}(s){d}s - \sum_{k = 1}^{N^{*}_{n}}\left(M^{\iota^{*}_{k-1}}_{\tau^{*}_{k}} - M^{\iota^{*}_{k-1}}_{\tau^{*}_{k-1}}\right) \hspace{1em} \mathsf{P}-\text{a.s.}
\end{align*}
 \end{proof}

\subsection{Convergence of the family of cumulative switching costs.}\label{Section:Optimal-Switching-ConvergenceOfSwitchingCostsInL2}
\subsubsection{A discrete-parameter martingale.}
For $k \ge 0$, define an $\mathcal{F}_{\tau^{*}_{k}}$-measurable random variable $\xi_{k}$ by,
\begin{equation}\label{eq:Optimal-Switching-MartingaleIncrements}
\xi_{k} \coloneqq \begin{cases}
M^{\iota^{*}_{k-1}}_{\tau^{*}_{k}} - M^{\iota^{*}_{k-1}}_{\tau^{*}_{k-1}} & \hspace{1em}\text{on } k \ge 1 \text{ and } \lbrace \tau^{*}_{k} < T \rbrace ,\\
0 & \hspace{1em}\text{otherwise.}
\end{cases}
\end{equation}
Note that the limit $\xi_{\infty}$ is a well-defined $\mathcal{F}_{T}$-measurable random variable which satisfies
\[
\xi_{\infty} \coloneqq \lim_{k}\xi_{k} = \begin{cases}
0, \hspace{1em} \text{on }\hspace{2bp} \lbrace N(\alpha^{*}) < \infty \rbrace,\\
0, \hspace{1em}\text{a.s. on }\hspace{2bp} \lbrace N(\alpha^{*}) = \infty \rbrace.
\end{cases} 
\]
where the second line holds since $M^{i}$, $i \in \mathbb{I}$, is quasi-left-continuous, and the switching times $\lbrace \tau^{*}_{k} \rbrace_{k \ge 1}$ announce $T$ on $\lbrace N(\alpha^{*}) = \infty \rbrace$ (cf. Lemma~\ref{Lemma:Optimal-Switching-FiniteStrategyCandidateOptimal}). In this case set $\iota^{*}_{\infty} \coloneqq \mathbf{u}^{*}_{T}$.

Since $M^{i} \in \mathcal{S}^{2}$ for $i \in \mathbb{I}$ (cf. Proposition~\ref{Proposition:Optimal-Switching-SquareIntegrableMartingale}) and the set $\mathbb{I}$ is finite, the sequence $\lbrace \xi_{k} \rbrace_{k \ge 0}$ is in $L^{2}$. Properties of square-integrable martingales and conditional expectations can be used to show:
\begin{align}\label{eq:Optimal-Switching-UniformBoundOnTheSquaredDifferences}
\forall n \ge 1, \quad \mathsf{E}\left[\sum_{k = 1}^{n} (\xi_{k})^{2}\right] & = \mathsf{E}\left[\sum_{k = 1}^{n}(M^{\iota^{*}_{k-1}}_{\tau^{*}_{k}} - M^{\iota^{*}_{k-1}}_{\tau^{*}_{k-1}})^{2}\mathbf{1}_{\{\tau^{*}_{k} < T\}}\right] \nonumber \\
& \le \sum_{i = 1}^{m}\sum_{k = 1}^{n}\mathsf{E}\left[\left((M^{i}_{\tau^{*}_{k}})^{2} - 2 \cdot M^{i}_{\tau^{*}_{k-1}} \cdot \mathsf{E}\big[M^{i}_{\tau^{*}_{k}} \bigm\vert \mathcal{F}_{\tau^{*}_{k-1}}\big] + (M^{i}_{\tau^{*}_{k-1}})^{2}\right)\right] \nonumber \\
& \le \sum_{i = 1}^{m}\mathsf{E}\left[(\sup\nolimits_{0 \le s \le T}|M^{i}_{s}|)^{2}\right] \nonumber \\
& \le 4 \cdot m \cdot \max_{i \in \mathbb{I}}\mathsf{E}\left[(M^{i}_{T})^{2}\right]
\end{align}
Finally, almost surely for $1 \le k \le N^{*}_{n}$,
\[
\mathsf{E}\bigl[\xi_{k} \bigm\vert \mathcal{F}_{\tau^{*}_{k-1}}\bigr] = \mathsf{E}\bigl[M^{\iota^{*}_{k-1}}_{\tau^{*}_{k}} - M^{\iota^{*}_{k-1}}_{\tau^{*}_{k-1}} \bigm\vert \mathcal{F}_{\tau^{*}_{k-1}}\bigr] = \sum_{i \in \mathbb{I}}\mathbf{1}_{\{\iota^{*}_{k-1} = i\}}\mathsf{E}\bigl[M^{i}_{\tau^{*}_{k}} - M^{i}_{\tau^{*}_{k-1}} \bigm\vert \mathcal{F}_{\tau^{*}_{k-1}}\bigr] = 0
\]
and letting $n \to \infty$ shows that $\mathsf{E}\big[\xi_{k} \bigm\vert \mathcal{F}_{\tau^{*}_{k-1}}\big] = 0$ for $k \ge 1$. Now define an increasing family of sub-$\sigma$-algebras of $\mathcal{F}$, $\mathbb{G} = \left(\mathcal{G}_{n}\right)_{n \ge 0}$, by $\mathcal{G}_{n} \coloneqq \mathcal{F}_{\tau^{*}_{n}}$. Applying Lemma~\ref{Lemma:Optimal-Switching-FiniteStrategyCandidateOptimal} and Proposition~\ref{Proposition:Optimal-Switching-qlcFiltrations} shows that
\[
\mathcal{G}_{\infty} \coloneqq \bigvee_{n}\mathcal{G}_{n} = \bigvee_{n}\mathcal{F}_{\tau^{*}_{n}} = \mathcal{F}_{T}.
\]
The sequence $\left(X_{n},\mathcal{G}_{n}\right)_{n \ge 0}$ with $X_{n}$ defined by
\begin{equation}\label{eq:Optimal-Switching-MartingaleFromMartingaleDifference}
X_{n} \coloneqq \sum_{k=0}^{n}\xi_{k}
\end{equation}
is a discrete-parameter martingale in $L^{2}$. The probability space $\left(\Omega,\mathcal{F},\mathsf{P}\right)$ with filtration $\mathbb{G} = \left(\mathcal{G}_{n}\right)_{n \ge 0}$ will be used to discuss convergence and integrability properties of $\left(X_{n}\right)_{n \ge 0}$.
\subsubsection{Convergence of the discrete-parameter martingale.}
As discussed previously, the $\mathbb{G}$-martingale $\left(X_{n}\right)_{n \ge 0}$ is in $L^{2}$. It is not hard to verify, by the conditional Jensen inequality for instance, that the sequence $\left(X^{2}_{n}\right)_{n \ge 0}$ is a positive $\mathbb{G}$-submartingale. By Doob's Decomposition (Proposition~\MakeUppercase{\romannumeral 7}-1-2 of \cite{Neveu1975}), $\left(X^{2}_{n}\right)_{n \ge 0}$ can be decomposed uniquely as
\begin{equation}\label{eq:Optimal-Switching-DoobDecomposition}
X_{n}^{2} = Q_{n} + R_{n}
\end{equation}
where $\left(Q_{n}\right)_{n \ge 0}$ is an integrable $\mathbb{G}$-martingale and $\left(R_{n}\right)_{n \ge 0}$ is an increasing process (starting from $0$) with respect to $\mathbb{G}$. Convergence of $\left(X_{n}\right)_{n \ge 0}$ depends on the properties of the compensator $\left(R_{n}\right)_{n \ge 0}$, and this is made more precise by the following proposition.

\begin{proposition}[\cite{Neveu1975}, Proposition~\MakeUppercase{\romannumeral 7}-2-3]\label{Proposition:Optimal-Switching-ConvergenceOfL2DiscreteParameterMartingales}
	Let $\left(X_{n}\right)_{n \ge 0}$ be a square-integrable $\mathbb{G}$-martingale such that (without loss of generality) $X_{0} = 0$, and $\left(R_{n}\right)_{n \ge 0}$ denote the increasing process associated with the $\mathbb{G}$-submartingale $\left(X^{2}_{n}\right)_{n \ge 0}$ by the Doob decomposition~\eqref{eq:Optimal-Switching-DoobDecomposition}. Then if $\mathsf{E}[R_{\infty}]< \infty$, the martingale $\left(X_{n}\right)_{n \ge 0}$ converges in $L^{2}$; furthermore, $\mathsf{E}[(\sup_{n \ge 0}|X_{n}|)^{2}] \le 4\mathsf{E}[R_{\infty}]$.
\end{proposition}

We can now prove the main result.

\begin{theorem}[Square-integrable cumulative switching costs]\label{Theorem:Optimal-Switching-SquareIntegrableSwitchingCosts}
	The sequence $\{ C^{\alpha^{*}}_{n}\}_{n \ge 1}$ converges in $L^{2}$ and also satisfies $\mathsf{E}\big[(\sup_{n}\big|C^{\alpha^{*}}_{n}\big|)^{2}\big] < \infty$.
\end{theorem}
\begin{proof}
Proposition~\ref{Proposition:Optimal-Switching-AdmissibilityCandidateOptimalStrategy} gave the following representation for the switching cost sum:
\begin{align}\label{eq:Optimal-Switching-CumulativeSwitchingCostUniformIntegrabilityCheck}
C^{\alpha^{*}}_{n} & = Y^{\iota^{*}_{N^{*}_{n}}}_{\tau^{*}_{N^{*}_{n}}} - Y^{\iota^{*}_{0}}_{\tau^{*}_{0}} +  \int_{\tau^{*}_{0}}^{\tau^{*}_{N^{*}_{n}}}\psi_{\mathbf{u}^{*}_{s}}(s){d}s - \sum_{k = 1}^{N^{*}_{n}}\left(M^{\iota^{*}_{k-1}}_{\tau^{*}_{k}} - M^{\iota^{*}_{k-1}}_{\tau^{*}_{k-1}}\right) \nonumber \\
& = Y^{\iota^{*}_{N^{*}_{n}}}_{\tau^{*}_{N^{*}_{n}}} - Y^{\iota^{*}_{0}}_{\tau^{*}_{0}} +  \int_{\tau^{*}_{0}}^{\tau^{*}_{N^{*}_{n}}}\psi_{\mathbf{u}^{*}_{s}}(s){d}s - X_{N^{*}_{n}} \hspace{1em} \mathsf{P}-\text{a.s.}
\end{align}

Since $N(\alpha^{*}) < \infty$ almost surely, the sequences $\{\tau^{*}_{N^{*}_{n}}\}_{n \ge 1}$ and $\{\iota^{*}_{N^{*}_{n}}\}_{n \ge 1}$ converge almost surely to $\tau^{*}_{N(\alpha^{*})} \le T$ and $\iota^{*}_{N(\alpha^{*})} = \mathbf{u}^{*}_{T}$ respectively. Noting that $Y^{i} \in \mathcal{S}^{2}$ and $\psi_{i} \in \mathcal{M}^{2}$ for every $i \in \mathbb{I}$, we can prove the claim by showing that the martingale $\left(X_{n}\right)_{n \ge 0}$ converges in $L^{2}$ and $\mathsf{E}[(\sup_{n \ge 0}|X_{n}|)^{2}] < \infty$. For this it suffices to prove the hypothesis of Proposition~\ref{Proposition:Optimal-Switching-ConvergenceOfL2DiscreteParameterMartingales}. Towards this end, we apply Fatou's Lemma to the increasing process $\left(R_{n}\right)_{n \ge 0}$ associated with the $\mathbb{G}$-submartingale $\left(X^{2}_{n}\right)_{n \ge 0}$ to get
\begin{equation}\label{eq:Optimal-Switching-FatouLemmaCompensator}
\mathsf{E}\left[R_{\infty}\right] \le \liminf_{n \to \infty}\mathsf{E}\left[R_{n}\right].
\end{equation}
For $n \ge 1$, the random variable $R_{n}$ can be decomposed as follows \cite[p.~148]{Neveu1975}:
\begin{equation}\label{eq:Optimal-Switching-CompensatorDecomposition}
R_{n} = \sum_{k = 0}^{n-1}R_{k+1} - R_{k} = \sum_{k = 0}^{n-1}\mathsf{E}\left[\left(X_{k+1}-X_{k}\right)^{2} \bigm \vert \mathcal{G}_{k}\right] = \sum_{k = 0}^{n-1}\mathsf{E}\left[\left(\xi_{k+1}\right)^{2} \bigm \vert \mathcal{G}_{k}\right].
\end{equation}
Using equation~\eqref{eq:Optimal-Switching-CompensatorDecomposition} in \eqref{eq:Optimal-Switching-FatouLemmaCompensator} and applying the tower property of conditional expectations leads to
\begin{equation}\label{eq:Optimal-Switching-FatouLemmaAndCompensatorDecomposition} 
\mathsf{E}\left[R_{\infty}\right] \le \liminf_{n \to \infty}\mathsf{E}\left[\sum_{k = 0}^{n-1}\left(\xi_{k+1}\right)^{2}\right].
\end{equation}
The inequalities leading up to~\eqref{eq:Optimal-Switching-UniformBoundOnTheSquaredDifferences} above show that the right-hand side of \eqref{eq:Optimal-Switching-FatouLemmaAndCompensatorDecomposition} is finite and we conclude by applying Proposition~\ref{Proposition:Optimal-Switching-ConvergenceOfL2DiscreteParameterMartingales}. \end{proof}

\section*{Acknowledgments}
This research was partially supported by EPSRC grant EP/K00557X/1. The author would like to thank his PhD supervisor J. Moriarty and colleague T. De Angelis for their feedback which led to an improved draft of the paper. The author also expresses his gratitude to others who commented on a previous version of the manuscript.



\end{document}